\documentclass[10pt, a4paper]{article}

\usepackage{cite}

\usepackage{enumitem}

\usepackage{mathtools}

\usepackage[amsmath,amsthm,thmmarks]{ntheorem}
\usepackage{amscd, amssymb, mathrsfs,amsfonts}
\usepackage{wasysym}
\usepackage{amsmath}

\usepackage[capitalise]{cleveref}
\crefformat{equation}{(#2#1#3)}

\theoremstyle{break} 
\newtheorem{lemma}{Lemma}[section]
\newtheorem{corollary}[lemma]{Corollary}

\newtheorem{theorem}[lemma]{Theorem}
\newtheorem{assumption}[lemma]{Assumption}

\theoremstyle{definition}
\newtheorem{example}[lemma]{Example}
\newtheorem{definition}[lemma]{Definition}
\newtheorem{remark}[lemma]{Remark}

\providecommand{\keywords}[1]
{
	\small	
	\textbf{\textbf{Key words---}} #1
}

\newcommand*\diff{\mathop{}\!\mathrm{d}}

\providecommand{\classifications}[1]
{
	\small	
	\textbf{\textbf{AMS subject classifications---}} #1
}

\title{On the Characterization of Generalized Derivatives for the Solution Operator of the Bilateral Obstacle Problem}

\author{Anne-Therese Rauls \thanks{Departments of Mathematics, TU Darmstadt, Dolivostr.\ 15, 64293 Darmstadt, Germany, anne-therese.rauls@tu-darmstadt.de},
Stefan Ulbrich \thanks{Departments of Mathematics, TU Darmstadt, Dolivostr.\ 15, 64293 Darmstadt, Germany, stefan.ulbrich@tu-darmstadt.de}}

\date{\today}

\begin{document}
	
	\maketitle

	\begin{abstract}
		 We consider optimal control problems for a wide class of bilateral obstacle
		 problems where the control appears in a possibly nonlinear source term.
		 The non-differentiability of the solution operator poses the main challenge
		 for the application of efficient optimization methods and the
		 characterization of Bouligand generalized derivatives of the solution
		 operator is essential for their theoretical foundation and numerical
		 realization. In this paper, we derive specific elements of
		 the Bouligand generalized differential if the control operator satisfies
		 natural monotonicity properties. We construct monotone sequences of
		 controls where the solution operator is G\^ateaux differentiable and
		 characterize the corresponding limit element of the Bouligand generalized
		 differential as being the solution operator of a Dirichlet problem on a quasi-open
		 domain. In contrast to a similar recent result for the unilateral obstacle
		 problem \cite{RaulsUlbrich}, we have to deal with an opposite monotonic behavior of the active
		 and strictly active sets corresponding to the upper and lower obstacle.
		 Moreover, the residual is no longer a nonnegative functional on
		 $H^{-1}$ and its representation as the difference
		 of two nonnegative Radon measures requires special care. This necessitates new proof
		 techniques that yield two elements of the Bouligand generalized
		 differential. Also for the unilateral case we obtain an additional element to that derived in \cite{RaulsUlbrich}.
	\end{abstract}

\\

\classifications{47J20, 49J40, 49K40, 49J52, 58C20, 58E35}

\section{Introduction}
We consider the optimal control of bilateral obstacle problems
\begin{align}
\label{BOP}
\text{Find } y \in K_\psi^\varphi: \quad \langle Ly-f(u),z-y\rangle_{H^{-1}(\Omega),H_0^1(\Omega)} \geq 0
\quad \forall\, z \in K_\psi^\varphi,
\end{align}
with 
\begin{align*}
K_\psi^\varphi:=\{z \in H_0^1(\Omega)\mid \psi \leq z \leq \varphi \text{ q.e.\ in } \Omega\}.
\end{align*}
Here, $\Omega \subseteq \mathbb{R}^d$ is an open, bounded set,
$L\in {\mathcal L}(H_0^1(\Omega),H^{-1}(\Omega))$ is a coercive and strictly
T-monotone operator. Here, strict T-monotonicity means that
\begin{align*}
\left\langle L(y-z),(y-z)_+\right\rangle_{H^{-1}(\Omega),H_0^1(\Omega)}>0 
\end{align*}
for all $y, z \in H_0^1(\Omega)$ 
with $(y-z)_+=\sup(0,y-z) \neq 0$, see \cite[p.\ 105]{Rodrigues}.
Furthermore, $f\colon U\to H^{-1}(\Omega)$ is a 
%\todo{locally reicht?} Lipschitz continuous,
continuously differentiable and monotone operator on a partially ordered Banach space
$U$. The detailed assumptions on $f$ and $U$ will be given below in
\cref{Assumption}.
We assume that the obstacles $\psi, \varphi \in H^1(\Omega)$ 
are such that $K_\psi^\varphi$ is nonempty.
In addition, 
for some results in this paper we require the following assumption.

\begin{assumption}
	\label{AssumptionObstacles}
	We consider lower and upper obstacles 
	$\psi, \varphi \in H^1(\Omega)\cap L^\infty(\Omega)$
	and assume there is a constant $c_\psi^\varphi>0$ such that
	$\varphi-\psi \geq c_\psi^\varphi$ holds a.e.\ in $\Omega$.
\end{assumption}

It is well known, see for example \cite{Barbu, Kinderlehrer}, that
for each $u \in U$, the variational inequality \cref{BOP} has a unique solution
and the solution operator 
$S_f=S_{\psi,f}^\varphi \colon U \to H_0^1(\Omega)$ is locally Lipschitz continuous.

The optimal control of obstacle problems and elliptic variational
inequalities has been studied by many authors, see for example
\cite{Barbu,Bergounioux,BergouniouxLenhart,Friedman,ItoKunisch,HarderWachsmuth,HintermullerKopacka,KarkkeinenKunischTarvainen,
	KunischWachsmuth,MeyerRademacherWollner,Mignot,MignotPuel,RaulsUlbrich,RaulsWachsmuth,SchielaWachsmuth}.
By using penalization, relaxation or regularization
approaches, optimality conditions have been derived in
\cite{Barbu,Bergounioux,BergouniouxLenhart,MignotPuel,ItoKunisch,HintermullerKopacka},
where \cite{BergouniouxLenhart} considers the obstacles for the bilateral
case as controls. Numerical solution methods based on these techniques have
been developed in
\cite{ItoKunisch,HintermullerKopacka,KarkkeinenKunischTarvainen,KunischWachsmuth,MeyerRademacherWollner,SchielaWachsmuth}.
Other approaches consider directly the nonsmooth solution operator.
Different optimality systems for the optimal control of the obstacle problem
are compared in \cite{HarderWachsmuth}. For the application of nonsmooth
optimization methods like bundle methods, the knowledge of at least one
element in the generalized differential of the objective function is
required.
For the finite dimensional obstacle problems, a characterization of the whole Clarke
subdifferential of the reduced objective function was obtained in
\cite{Haslinger}.
The directional differentiability of solution operators of elliptic
variational inequalities and a variational inequality for the directional
derivative have been obtained in \cite{Haraux,Mignot}, see also
\cite{ChristofWachsmuth1}.
The resulting structure of G\^ateaux derivatives in points of differentiability was used
in \cite{RaulsWachsmuth} to characterize the full Bouligand generalized
differential for the solution operator of the unilateral obstacle problem
with distributed control $f(u)=u\in H^{-1}(\Omega)$. Subsequently, for
Lipschitz continuous, continuously differentiable, and monotone control
operators $f \colon U\to H^{-1}(\Omega)$ on a partially ordered Banach space
$U$, an element of the Bouligand generalized
differential for the solution operator of the unilateral obstacle problem
has been characterized in \cite{RaulsUlbrich} as the solution operator
of a variational equation on the inactive set. This forms an analytical
foundation to develop error estimators for the numerical computation of
subgradients and to apply inexact bundle methods in Hilbert space, see for
example \cite{HertleinUlbrich}.

In this paper, we derive, based on a characterization of the G\^ateaux
derivative in points of differentiability,
two elements of the Bouligand generalized
differential for the solution operator $S_f\colon U \to H_0^1(\Omega)$
of the bilateral obstacle problem \cref{BOP} in points of nonsmoothness.
To this end, we require the following monotonicity property of the control operator
$f\colon U \to H^{-1}(\Omega)$.

\begin{assumption}
	\label{Assumption}
	We assume that the operator $f \colon U \to H^{-1}(\Omega)$ is defined on a partially ordered Banach space $(U,\geq_U)$.
	Let $f$ be increasing, i.e.,
	$u_1 \geq_U u_2$ implies $f(u_1) \geq f(u_2)$
	in $H^{-1}(\Omega)$,
	i.e., 
	$\langle f(u_1)-f(u_2),v\rangle_{H^{-1}(\Omega),H_0^1(\Omega)} \geq 0$
	for all $v \in H_0^1(\Omega)_+:=\{v \in H_0^1(\Omega)\mid v \geq 0\}$.
	Moreover, let $f$ be continuously differentiable.
	
	In addition, we assume that $U$ is separable
	and that there is a partially ordered Banach space $(V,\geq_V)$
	such that the positive cone $\mathcal{P}=\{v \in V\mid v \geq_V 0\}$
	has nonempty interior
	and $V$ is embedded into $U$.
	The order relation $\geq_V$
	has the property that for all $v_1, v_2 \in V$
	with $v_1 \geq_V v_2$
	we have $v_1+z\geq_V v_2+z$ for all $z \in V$
	and $t\,v_1 \geq_V t\,v_2$ for all scalars $t \geq 0$.
	We assume that the linear embedding $\iota \colon V \to U$
	is continuous, dense and increasing, i.e.,
	compatible with the order structures in $V$ and $U$.
	This means that $v \in V$ with $v \geq_V 0$
	implies $\iota(v)\geq_U 0$ in $U$.
\end{assumption}

With the help of a generalization of Rademacher's theorem to infinite dimensions,
see \cref{Thm:Rademacher},
the assumption allows us to construct for
each $u\in U$ a monotonically increasing (or monotonically decreasing) sequence
$(u_n)_{n \in \mathbb{N}}\subseteq U$ converging to $u$, such that the locally Lipschitz continuous solution operator $S_f$
is G\^ateaux differentiable at $u_n$.
In particular, the assumptions on $U$ are fulfilled for $U=L^2(\Omega)$,
$U=H^{-1}(\Omega)$ and $U=\mathbb{R}^k$.

Our strategy for computing an element of the Bouligand generalized
differential is as follows. Let $u\in U$ and denote by $I(f(u))$ and
$A(f(u)):=A_\psi(f(u))\cup A^\varphi(f(u))$ the
inactive and active set of the solution $S_f(u)$, respectively. By
representing the residual $L S_f(u)-f(u)\in H^{-1}(\Omega)$ as the difference
$\tilde{\xi}_\psi-\tilde{\xi}^\varphi$ of two nonnegative Radon measures (\cref{Lem:SignedMeasure}), we can
define the strictly active set $A_{\mathrm{s}}(f(u)):=A^{\mathrm{s}}_\psi(f(u))\cup A_{\mathrm{s}}^\varphi(f(u))$
using the measures $\tilde{\xi}_\psi$ and $\tilde{\xi}^\varphi$.

If the solution operator $S_f$ is G\^ateaux differentiable at $u$, then
$S_f'(u)$ can be obtained as the solution operator of a variational equation
on $H_0^1(D)$, where $D$ can be chosen as any quasi-open subset of
$\Omega$ satisfying $I(f(u)) \subseteq D \subseteq \Omega \setminus A_{\mathrm{s}}(f(u))$,
see \cref{Theo:GateauxDerivative}.

If the solution operator $S_f$ is nonsmooth at $u$, then we can construct 
a monotonically increasing sequence
$(u_n)_{n \in \mathbb{N}}\subseteq U$, where $S_f$ is G\^ateaux differentiable, converging to $u$
and $S_f'(u_n)$ can thus be represented as the solution operator of a variational equation
on $H_0^1(D_n)$ with $D_n:=I(f(u_n))\cup (A^\varphi(f(u_n))\setminus
A_{\mathrm{s}}^\varphi(f(u_n)))$. By using monotonicity properties of the sequence of sets
$(A_{\mathrm{s}}^\varphi(f(u_n)))_{n \in \mathbb{N}}$ and $(A_\psi(f(u_n)))_{n \in \mathbb{N}}$ we show that $(H_0^1(D_n))_{n \in \mathbb{N}}$ converges in the sense of
Mosco to $H_0^1(D)$
for $D:=I(f(u))\cup (A^\varphi(f(u))\setminus
A_{\mathrm{s}}^\varphi(f(u)))$, 
cf.\ \cref{Lem:MoscoConvergence}, and stability properties of variational equations
yield that $(S_f'(u_n))_{n \in \mathbb{N}}$ converges in the strong operator topology to an element
in the Bouligand generalized differential of $S_f$ at $u$, which can be
characterized as the solution operator of a variational equation on
$H_0^1(D)$, see \cref{Theo:MainResult}.
Working with a monotonically decreasing sequence $(u_n)_{n \in \mathbb{N}}\subseteq U$ yields another
generalized derivative.

In \cite{RaulsUlbrich}, a similar approach has been used for the unilateral obstacle
problem. The analysis for the bilateral obstacle problem is more involved and requires new proof techniques for several reasons. First of all, to the best of our knowledge, the decomposition of $L S_f(u)-f(u)\in H^{-1}(\Omega)$ as the difference
$\tilde{\xi}_\psi-\tilde{\xi}^\varphi$ of two nonnegative Radon measures has not been established so far and requires
care. This representation is, in general, only valid on $H_0^1(\Omega)\cap L^\infty(\Omega)$, see \cref{Lem:SignedMeasure}, or, alternatively, if the
distance of the active sets $A^\varphi(f(u))$ and $A_\psi(f(u))$ is positive (\cref{Lem:DistanceActiveSets}), which
we demonstrate by giving a counter example in \cref{ex:counterexample}. Next, while for $(u_n)_{n \in \mathbb{N}}$
increasing the choice $D_n:=I(f(u_n))$ is 
possible and monotonically increasing in the unilateral case, 
the sets $D_n:=I(f(u_n))\cup (A^\varphi(f(u_n))\setminus A_{\mathrm{s}}^\varphi(f(u_n)))$ do
not enjoy monotonicity properties and require the more difficult study of
monotonicty properties of the strictly active sets $A_{\mathrm{s}}^\varphi(f(u_n))$,
see \cref{Lem:MonotonicityStrictlyActiveSets}.
Also in the unilateral case,
the analysis in this paper yields
an additional element of the Bouligand generalized differential
to that derived in \cite{RaulsUlbrich}.

If $L$ is induced by a symmetric coercive bilinear form, then it
is well known that \cref{BOP} are first order optimality conditions of the
problem
\[
\min_{y\in H_0^1(\Omega)} \langle \tfrac{1}{2}Ly-f(u),y\rangle_{H^{-1}(\Omega),H_0^1(\Omega)}
~~\mbox{subject to } y\in K_\psi^\varphi.
\]
Then the residual $L S_f(u)-f(u)\in H^{-1}(\Omega)$ corresponds to the
Lagrange multiplier and our careful study of its representation as the difference
$\tilde{\xi}_\psi-\tilde{\xi}^\varphi$ of two nonnegative Radon measures
gives detailed insights into the structure of the Lagrange multiplier.
This might be helpful also in other contexts, for example the design
and analysis of efficient solution methods for \cref{BOP}.

The paper is organized as follows. \Cref{Sec:Def} recalls some basic
definitions and results on capacity theory, Sobolev spaces on quasi-open
sets and generalized derivatives. 
In \Cref{Sec:DifferentiabilityProperties}, monotonicity and differentiability
properties of the solution operator are analyzed. The well known variational
inequality for the directional derivative is stated and the structure of the
critical cone is studied. To this end, a representation of the residual
$LS_f(u)-f(u)$ as the difference $\tilde{\xi}_\psi-\tilde{\xi}^\varphi$ of two nonnegative
Radon measures is derived and the strictly active sets are defined based on
these measures. A counterexample shows that this representation
is, in general, only valid on $H_0^1(\Omega)\cap L^\infty(\Omega)$, which
requires some care in the sequel. Moreover, a representation of $S_f'(u)$ is
derived if $u$ is a point of G\^ateaux differentiability. In
\Cref{Sec:Monotonicity} the monotonicity of the strictly active and
active sets $A_{\mathrm{s}}^\varphi(f(u_n))$ and $A_\psi(f(u_n))$ 
is analyzed for monotonically increasing control sequences $(u_n)$ 
(respectively, of $A^{\mathrm{s}}_\psi(f(u_n))$ and $A^\varphi(f(u_n))$ for monotonically decreasing
$(u_n)_{n \in \mathbb{N}}$).
This is used in \Cref{Sec:Mosco} to show the Mosco convergence of
$(H_0^1(I(f(u_n))\cup (A^\varphi(f(u_n))\setminus A_{\mathrm{s}}^\varphi(f(u_n)))))_{n \in \mathbb{N}}$
and $(H_0^1(I(f(u_n))\cup (A_\psi(f(u_n))\setminus A^{\mathrm{s}}_\psi(f(u_n)))))_{n \in \mathbb{N}}$, respectively.
\Cref{Sec:GenDer} uses now stability properties of variational equalities
under Mosco convergence to characterize elements of the Bouligand
generalized differential. Finally, an adjoint
representation of corresponding Clarke subgradients for an objective
functional is derived.

\section{Fundamental Definitions and Results}\label{Sec:Def}

Denote by $C_{\mathrm{c}}(\Omega)$ the space of continuous functions on $\Omega$
with compact support contained in $\Omega$
and by $C_{\mathrm{c}}^\infty(\Omega)$ the subspace of infinitely differentiable functions.
Furthermore,
we denote by $H^1(\Omega)$ the space
\begin{align*}
H^1(\Omega):=\left\{z \in L^2(\Omega)\mid \tfrac{\partial z}{\partial x_i} \in L^2(\Omega), i=1,\ldots,d\right\},
\end{align*}
where $\tfrac{\partial z}{\partial x_i}$ is to be understood in the distributional sense.
$H^1(\Omega)$ is equipped with the norm
\begin{align*}
\|z\|_{H^1(\Omega)}=\left(\int_{\Omega} z^2+\sum_{i=1}^d \left(\frac{\partial z}{\partial x_i}\right)^2\,\diff\lambda^d\right)^{1/2}.
\end{align*}
In this paper,
we work with the space $H_0^1(\Omega)$
and we define it as the completion of $C_{\mathrm{c}}^\infty(\Omega)$
in $H^1(\Omega)$.
On $H_0^1(\Omega)$ we consider the norm
$\|z\|_{H_0^1(\Omega)}:=\|\nabla z\|_{L^2(\Omega)}$.
The dual space of $H_0^1(\Omega)$ is denoted by $H^{-1}(\Omega)$
and if $\zeta \in H^{-1}(\Omega)$ and $z \in H_0^1(\Omega)$
we use the notation $\langle \zeta,z\rangle$ for the dual pairing.

Note that $z \in H_0^1(\Omega)$ can be 
extended by zero to an element of $z \in H^1(\mathbb{R}^d)$,
since the zero extension of the approximating sequence in $C_{\mathrm{c}}^\infty(\Omega)$
is a Cauchy sequence in $H^1(\mathbb{R}^d)$.

We denote by $H^1(\Omega)_+$, respectively $H_0^1(\Omega)_+$,
the respective subsets of nonnegative elements.
For $u\in H_0^1(\Omega)$ and $v \in H^1(\Omega)_+$ we often use that
$\min(u,v) \in H_0^1(\Omega)$.
We use the notation $u_+:=\max(0,u)$, $u_-=-\min(0,u)$,
for the positive and negative part of $u \in L^2(\Omega)$
and have $u=u_+-u_-$.
Furthermore, 
for $n \in \mathbb{N}$ and $u \in H_0^1(\Omega)$,
$u_n:=\max(-n,\min(u,n))$ is in $H_0^1(\Omega)\cap L^\infty(\Omega)$
and it holds $u_n \to u$ in $H_0^1(\Omega)$,
which can be seen by application of Lebesgue's dominated convergence theorem.

\subsection{Capacity Theory}

We quickly recall and clarify the definitions and concepts related to capacity theory
that we consider in this paper.
For the definitions, see also \cite[Sect.~5.82,~5.83]{AttouchButtazzoMichaille}, \cite[Def.~6.2]{DelfourZolesio},
\cite{KilpelainenMaly}.

\begin{definition}
	\label{Def:Capacity}
	\begin{enumerate}
		\item 
		For a set $E \subseteq \Omega$ we define the capacity of $E$ in $\Omega$ by
		\begin{align}
		\label{capacity}
		\operatorname{cap}(E):=\inf\{\|z\|^2_{H_0^1(\Omega)}\mid z \in H_0^1(\Omega), z \geq 1 \text{ a.e.\ in a neighborhood of } E\}.
		\end{align}
		If a property holds on a set $E \subseteq \Omega$
		except on a subset of capacity zero,
		we say that this property holds quasi-everywhere (q.e.) on $E$.
		\item 
		We call a set $O \subseteq \Omega$ quasi-open
		if for all $\varepsilon>0$ there is an open set $O_\varepsilon \subseteq \Omega$
		such that $O \cup O_\varepsilon$ is open and  $\operatorname{cap}(O_\varepsilon)<\varepsilon$.
		A set $A \subseteq \Omega$ is quasi-closed
		if the complement in $\Omega$ is quasi-open.
		\item 
		Let $v \colon \Omega \to \mathbb{R}$ be a function.
		Then $v$ is quasi-continuous
		if for all $\varepsilon>0$
		there exists an open set $O_\varepsilon \subseteq \Omega$
		such that $v|_{\Omega \setminus O_\varepsilon}$ is continuous
		and $\operatorname{cap}(O_\varepsilon)<\varepsilon$.
		\item 
		Let $ E\subseteq \Omega$ be a set.
		A family $(O_i)_{i \in I}$ of quasi-open subsets of $\Omega$
		is called a quasi-covering of $E$
		if there is a countable subfamily $(O_{i_n})_{n \in \mathbb{N}}$
		satisfying
		$\operatorname{cap}\left(E \setminus \bigcup_{n \in \mathbb{N}}O_{i_n}\right)=0$.
		\item 
		Let $O \subseteq \Omega$ be quasi-open.
		Then we define
		\begin{align*}
		H_0^1(O):=\{z \in H_0^1(\Omega)\mid z=0 \text{ q.e.\ on } \Omega \setminus O\}.
		\end{align*}
	\end{enumerate}
\end{definition}

\begin{remark}
	If $O \subseteq \Omega$ is open,
	then the definition of $H_0^1(O)$
	in \cref{Def:Capacity}
	coincides with the classical definition of $H_0^1(O)$
	that we also use in this paper,
	see e.g.\ \cite[Thm.~9.1.3]{AdamsHedberg}. 
	Moreover,
	the definition of $H_0^1(O)$ for $O \subseteq \Omega$ quasi-open
	coincides with the definition 
	\begin{align*}
	H_0^1(O)=\bigcap \{H_0^1(G)\mid O \subseteq G \subseteq \Omega, G \text{ open}\}
	\end{align*}
	given in \cite{KilpelainenMaly}
	(up to extension of the elements by $0$).
	%	Note that our convention is to consider
	%	all $H_0^1$-spaces as subspaces of $H^1(\mathbb{R}^d)$
	%	with values on all of $\mathbb{R}^d$
	%	to simplify the notation. 
	
	We could also define a capacity $\operatorname{Cap}$ for subsets 
	of $\mathbb{R}^d$
	by testing with $H^1(\mathbb{R}^d)$-elements in \cref{capacity}
	and by considering the infimum over the squared $H^1(\mathbb{R}^d)$-norms.
	Then $H_0^1(U)=\{v \in H^1(\mathbb{R}^d)\mid v=0\, \operatorname{Cap}\text{-q.e.\ outside } U\}$
	for $\operatorname{Cap}$-quasi-open subsets $U$ of $\mathbb{R}^d$.
	For $U \subseteq \Omega$,
	$U$ is $\operatorname{Cap}$-quasi-open if and only if it is
	quasi-open
	and the definitions of $H_0^1(U)$ coincide.
\end{remark}

\begin{lemma}
	\label{Lem:Capacity}
	\begin{enumerate}
		\item 
		Let $v \in H^1(\Omega)$. 
		Then $v$ has a quasi-continuous representative $\tilde{v}$.
		If $\tilde w$ is another quasi-continuous representative of $v$,
		then $\tilde v=\tilde w$ up to a set of capacity zero. 
		\item 
		Suppose $(v_n)_{n \in \mathbb{N}}, v \subseteq H_0^1(\Omega)$
		and $v_n \to v$ in $H_0^1(\Omega)$.
		Then there is a subsequence
		$(v_{n_k})_{k \in \mathbb{N}}$ such that $\tilde v_{n_k} \to  \tilde v$ 
		pointwise q.e.\ for the quasi-continuous representatives.
		\item
		Let $O \subseteq \Omega$ be a quasi-open set,
		let $z \in H_0^1(O)$
		and assume that $(O_i)_{i \in I}$ is a quasi-covering of $O$.
		Then there exists a sequence $(z_n)_{n \in \mathbb{N}} \subseteq H_0^1(O)$
		such that $z_n \to z$ 
		and such that each $z_n$ is a finite sum of functions from
		$\bigcup_{i \in I} H_0^1(O_i)$.
		If $(O_n)_{n \in \mathbb{N}}$ is a quasi-covering
		that is increasing in $n$,
		then we find a sequence $(z_n)_{n \in \mathbb{N}}$
		converging to $z$
		such that $z_n \in H_0^1(O_n)$.
		If $z \in H_0^1(\Omega)_+$,
		then, w.l.o.g., $(z_n)_{n \in \mathbb{N}} \subseteq H_0^1(\Omega)_+$.
		\item 
		Assume $O \subseteq \Omega$ is quasi-open
		and suppose $v \colon \Omega \to \mathbb{R}$ is quasi-continuous.
		Then, $v \geq 0$ a.e.\ on $O$ if and only if $v \geq 0$ q.e.\ on $O$.
		\item 
		Assume $O \subseteq \Omega$ is a quasi-open set. 
		Then there exists $v \in H_0^1(\Omega)_+$ with $\{\tilde v>0\}=O$
		up to a set of zero capacity.
		\item After modification on a subset of capacity zero,
		each quasi-continuous function $v\colon \Omega \to \mathbb{R}$
		is Borel measurable.
	\end{enumerate}
\end{lemma}

\begin{proof}
	The first statement can be found, e.g., in \cite[Chap.~6,~Thm.~6.1]{DelfourZolesio} or \cite[Thm.~4.4]{HeinonenKilpelaeinenMartio},
	the second in \cite[Lem.~6.52]{BonnansShapiro}.
	The first part of the third statement can be obtained by combining \cite[Lem.~2.4~and~Lem.~2.10]{KilpelainenMaly}. 
	For the second part of the third statement
	use that $(O_n)_{n \in \mathbb{N}}$ is increasing and use the first part.
	In case $z \geq 0$, the statement in \cite[Lem.~2.4]{KilpelainenMaly} and the proof of \cite[Lem.~2.10]{KilpelainenMaly} imply that we can choose $(z_n)_{n \in \mathbb{N}} \subseteq H_0^1(\Omega)_+$.
	We refer to \cite[Lem.~2.3]{WachsmuthStrongStationarity} for the fourth statement and to \cite[Prop.~2.3.14]{Velichkov}, \cite[Lem. 3.6]{HarderWachsmuth}
	for the fifth statement.
	The last statement can be found in \cite[Sec.~2.1]{RaulsWachsmuth}.
\end{proof}

\begin{remark}
	%	\begin{enumerate}
	%		\item 
	Let $v \colon \Omega \to \mathbb{R}$ be a quasi-continuous function.
	Then the set $\{v>0\}$ is quasi-open
	and the set $\{v \geq 0\}$ is quasi-closed.
	If $v \in H^1(\Omega)$,
	then by $\{v>0\}, \{v\geq0\}$ we always mean the sets
	$\{\tilde{v}> 0\}, \{\tilde v\geq0\}$,
	where $\tilde{v}$ is a quasi-continuous and Borel measurable representative.
	Thus, these sets are Borel measurable, quasi-open, respectively quasi-closed, and determined up to a set of capacity zero.
	%		\item
	%		
	%	\end{enumerate}
\end{remark}

Throughout the rest of the paper,
when considering set equations or inclusions
for subsets of $\Omega$,
they have to be understood to hold up to an exceptional set of capacity zero.

\subsection{Generalized Derivatives}

We consider the following generalized differential for the solution operator of \cref{BOP}.
\begin{definition}
	\label{Def:Differential}
	Consider a separable Banach space $X$ 
	and a Hilbert space $Y$.
	Assume that $T \colon X \to Y$ is a locally Lipschitz continuous
	operator.
	The set of Bouligand generalized derivatives of $T$ in $x \in X$ is defined as 
	\begin{align}
	\label{GeneralizedDifferential}
	\begin{split}
	\partial T(x):=\left\{\Xi \in \mathcal{L}(X,Y)\mid \exists (x_n)_{n \in \mathbb{N}} \subseteq \mathcal{D}_T \text{ with } x_n \to x \right. \\ \left. \text{and } T'(x_n) \to \Xi \text{ in the weak operator topology}\right\},
	\end{split}
	\end{align}
	for $\mathcal{D}_T:=\{x \in X\mid T \text{ is G\^ateaux differentiable at } x \text{ with G\^ateaux derivative } T'(x)\}$.
\end{definition}

For infinite dimensional spaces,
there are several choices of topologies in $X$ and $Y$.
Combinations of strong and weak topologies in $X$ and $Y$
lead to four possible definitions that do not coincide, in general.
The four versions are defined,
e.g.\ in \cite[Def.~3.1]{ChristofMeyerWaltherClason}, \cite[Def.~2.10]{RaulsWachsmuth},
and characterized for the solution operators of a nonsmooth semilinear elliptic equation
and the unilateral obstacle problem with distributed controls,
respectively.
The set $\partial T(x)$ as defined in \cref{Def:Differential}
is always nonempty as a consequence of Rademacher's theorem
for locally Lipschitz continuous mappings.
Nevertheless, 
the generalized derivatives we construct in this paper
are also contained in the generalized differential that can be obtained
by replacing the weak operator topology in \cref{GeneralizedDifferential}
by the strong operator topology.
A priori it is not clear
that this set is nonempty.

\begin{remark}
	\label{Rem:InclusionGeneralizedDifferentials}
	Let $J \colon H_0^1(\Omega) \times U \to \mathbb{R}$
	be a continuously differentiable objective function.
	Denote by $S_f \colon U \to H_0^1(\Omega)$ the solution operator
	of \cref{BOP}.
	We use the notation $\hat{J}:=J(S_f(\cdot),\cdot)$ 
	for the reduced objective function
	and denote by $\partial_C \hat{J}$ Clarke's generalized differential of $\hat{J} \colon U \to \mathbb{R}$.
	Let $u \in U$ be arbitrary.
	Then the set inclusion
	\begin{align*}
	\{\Xi^*J_y(S_f(u),u)+J_u(S_f(u),u)\mid \Xi \in \partial S_f(u)\}
	\subseteq \partial \hat{J}(u)
	\subseteq \partial_C \hat{J}(u)
	\end{align*}
	holds.
\end{remark}

\section{Properties of the Solution Operator}
\label{Sec:DifferentiabilityProperties}
In this section,
we collect properties of the solution operator $S_f$ of \cref{BOP}.

\subsection{Monotonicity Properties of the Solution Operator}

We state two lemmata on monotonicity properties of $S_f$.
The next lemma summarizes the monotonicity 
of $S_f$ with respect to the elements in $U$. 
The proof is similar to the proof of \cite[Sect.~4:5,~Thm.~5.1]{Rodrigues}
where the property is shown for the unilateral obstacle problem.

\begin{lemma}
	\label{Lem:Monotonicity}
	Let $u_1, u_2 \in U$ with $u_1 \geq_U u_2$.
	Then $S_f(u_1) \geq S_f(u_2)$ a.e.\ and q.e.\ in $\Omega$.
\end{lemma}

\begin{proof}
	For $i=1,2$ set $y_i:=S_f(u_i)$.
	We test the variational inequality characterizing $y_1$ with
	$z_1=\max(y_1,y_2)=y_1+(y_2-y_1)_+$
	and the variational inequality characterizing $y_2$ with
	$z_2=\min(y_1,y_2)=y_2-(y_2-y_1)_+$,
	respectively,
	and obtain
	\begin{align*}
	0 \leq \langle Ly_1-f(u_1),z_1-y_1\rangle=\langle Ly_1-f(u_1),(y_2-y_1)_+\rangle
	\end{align*}
	and
	\begin{align*}
	0 \leq \langle Ly_2-f(u_2),z_2-y_2\rangle =\langle Ly_2-f(u_2),-(y_2-y_1)_+\rangle.
	\end{align*}
	Summing up both inequalities we obtain
	\begin{align*}
	\langle Ly_1-Ly_2,(y_2-y_1)_+\rangle\geq\langle f(u_1)-f(u_2),(y_2-y_1)_+\rangle\geq 0.
	\end{align*}
	By strict T-monotonicity, we have $(y_2-y_1)_+=0$, i.e.,
	$y_1 \geq y_2$ a.e.\ and q.e.\ in $\Omega$.
\end{proof}

The following lemma establishes monotonicity properties of $S_f$ with respect to one of the obstacles.
\begin{lemma}\label{lem:monobst}
	Let $\psi_i\in H^1(\Omega)$, $i=1,2$, such that $K_{\psi_i}^\varphi$
	are nonempty, and denote by $y_i$ the corresponding solutions of \cref{BOP} for fixed $u \in U$. Then
	$\psi_1\geq\psi_2$ implies $y_1\geq y_2$ a.e.\ and q.e.\ in $\Omega$.
\end{lemma}
\begin{proof}
	We test the variational inequality characterizing $y_1$ with the element
	$z_1=\max(y_1,y_2)=y_1+(y_2-y_1)_+ \in K_{\psi_1}^\varphi$
	and the variational inequality characterizing $y_2$
	with $z_2=\min(y_1,y_2)=y_2-(y_2-y_1)_+\in K_{\psi_2}^\varphi$,
	respectively,
	and obtain
	\[
	0 \leq \left\langle Ly_1-f(u),(y_2-y_1)_+ \right\rangle,\quad
	0 \leq \left\langle Ly_2-f(u),-(y_2-y_1)_+ \right\rangle.
	\]
	Summing up both inequalities we obtain
	\[
	\left\langle L(y_1-y_2),(y_2-y_1)_+ \right\rangle \geq 0.
	\]
	By strict T-monotonicity
	we have $(y_2-y_1)_+=0$,
	i.e., $y_1 \geq y_2$ a.e.\ and q.e.\ in $\Omega$.
\end{proof}

\subsection{Differentiability Properties of the Solution Operator}
\label{Subsec:DifferentiabilityProperties}

We distinguish the following subsets of $\Omega$
for a fixed element $\zeta \in H^{-1}(\Omega)$
that result from the solution $S_{\operatorname{id}}(\zeta)$ of \cref{BOP}
for $f=\operatorname{id}$ being the identity operator on $H^{-1}(\Omega)$.
Let $\zeta \in H^{-1}(\Omega)$.
By 	
\begin{align*}
A(\zeta):=\{\omega \in \Omega\mid S_{\operatorname{id}}(\zeta)(\omega)=\psi(\omega) \text{ or } S_{\operatorname{id}}(\zeta)(\omega)=\varphi(\omega)\}
\end{align*}	
we denote the active set.
We also distinguish the active sets with respect to
$\psi$ and $\varphi$,
i.e.,
we define
\begin{align*}
A_\psi(\zeta):=\{\omega \in \Omega\mid S_{\operatorname{id}}(\zeta)(\omega)=\psi(\omega)\}
\quad \text{ and } \quad 
A^\varphi(\zeta):=\{\omega \in \Omega\mid S_{\operatorname{id}}(\zeta)(\omega)=\varphi(\omega)\}.
\end{align*}
Note that $A(\zeta)=A_\psi(\zeta)\cup A^\varphi(\zeta)$
and that all these sets are quasi-closed sets
that are determined up to a set of capacity zero,
since we consider quasi-continuous representatives of $S_{\operatorname{id}}(\zeta), \psi, \varphi \in H^1(\Omega)$ in the definition of the active sets.
We denote by $I(\zeta)$ the inactive set, i.e.\ the complement of $A(\zeta)$ in $\Omega$
and by $I_\psi(\zeta):=\Omega \setminus A_\psi(\zeta)$, respectively 
$I^\varphi(\zeta):=\Omega \setminus A^\varphi(\zeta)$,
the inactive sets with respect to the two obstacles.

Let $u, h \in U$.
It can be shown
that the directional derivative $S_f'(u;h)$ of the solution operator of variational inequality \cref{BOP} is given by
the solution of
\begin{align}
\label{DirectionalDerivative}
\begin{split}
&\text{Find } \eta \in (LS_f(u)-f(u))^\perp\cap T_{K_\psi^\varphi}(S_f(u)):\\
&\langle L\eta-f'(u;h),z-\eta\rangle\geq0 \quad \forall\,z \in (LS_f(u)-f(u))^\perp\cap T_{K_\psi^\varphi}(S_f(u)).
\end{split}
\end{align}
Here,
$T_{K_\psi^\varphi}(S_f(u))$ 
is the tangent cone of $K_\psi^\varphi$ at $S_f(u)$,
i.e., the closed conic hull of $K_\psi^\varphi-S_f(u)$.
When $f$ is the identity operator on $H^{-1}(\Omega)$,
the variational inequality \cref{DirectionalDerivative}
follows, e.g., from \cite[Thm.~3.3]{Mignot}.
Since $S_f$ is locally Lipschitz continuous,
$S_f$ is even directionally differentiable in the sense of Hadamard,
see \cite[Prop.~2.49]{BonnansShapiro}.
When considering a general operator $f \colon U \to H^{-1}(\Omega)$
fulfilling our assumptions,
\cref{DirectionalDerivative} can be obtained
using the chain rule for Hadamard directionally differentiable
maps, see e.g.\ \cite[Prop.~2.47]{BonnansShapiro}.

By \cite[Lem.~3.4]{Mignot},
we have 
\begin{align}
\label{Tan}
T_{K_\psi^\varphi}(S_f(u))=\{z \in H_0^1(\Omega)\mid z \geq 0 \text{ q.e.\ on } A_\psi(f(u)), z \leq 0 \text{ q.e.\ on } A^\varphi(f(u))\}.
\end{align}

\subsubsection{Analysis of the Critical Cone}

As in the case with a single obstacle, 
we want to find a suitable characterization of the critical cone $(LS_{\operatorname{id}}(\zeta)-\zeta)^\perp \cap T_{K_\psi^\varphi}(S_{\operatorname{id}}(\zeta))$
for arbitrary $\zeta \in H^{-1}(\Omega)$.
Note that in the case with a single lower obstacle such a characterization is given by $\{z \in H_0^1(\Omega)\mid z \geq 0 \text{ q.e.\ on } A(\zeta), z=0 \text{ q.e.\ on } A_{\mathrm{s}}(\zeta)\}$, see \cite[Lem.~3.1]{WachsmuthStrongStationarity}.
Here, $A_{\mathrm{s}}(\zeta)$ is the strictly active set,
which can also be characterized as in \cite[App.~A]{WachsmuthStrongStationarity}.

A crucial difference to the case with only one obstacle
is that $LS_{\operatorname{id}}(\zeta)-\zeta$ is not a nonnegative functional
and thus cannot be identified with a positive measure.
Instead,
we will see that, in some cases, it can be identified with the difference of
two nonnegative Radon measures.
In general, 
i.e., when the active sets $A_\psi(\zeta)$ and $A^\varphi(\zeta)$ 
do not have a positive distance,
$LS_{\operatorname{id}}(\zeta)-\zeta$ acts as the difference of two measures on all elements of $H_0^1(\Omega) \cap L^\infty(\Omega)$,
but the characterization does not carry over to unbounded elements of
$H_0^1(\Omega)$, see \cref{ex:counterexample}.

We define the set of nonnegative Radon measures $\mathcal{M}_+(\Omega)$ on $\Omega$ as
\begin{align*}
\mathcal{M}_+(\Omega)=\{\mu\mid \, &\mbox{$\mu$ 
	is a 
	regular, locally finite Borel measure on $\Omega$}\}.
\end{align*}
Before we present the theorem on the representation,
let us state an auxiliary lemma.

\begin{lemma}
	\label{Lem:ApplicationDominatedConvergence}
	Assume that $v \in L^\infty(\Omega) \cap H^1_0(\Omega)$. Moreover, let
	$(w_n)_{n \in \mathbb N} \subseteq L^\infty(\Omega) \cap H^1_0(\Omega)$ be
	a sequence with $w_n \to 0$ in $H^1_0(\Omega)$ and $|w_n|\leq C$ a.e. for
	some $C>0$ and all $n \in \mathbb N$. Then $v\, w_n \to 0$ in
	$H^1_0(\Omega)$.
\end{lemma}

\begin{proof}
	Let us recall that $v\,w_n \in H_0^1(\Omega)$
	and $\nabla(v\,w_n)=w_n\nabla v+v\nabla w_n$,
	see, e.g., \cite[(7.18)]{GilbargTrudinger}.
	Thus,
	\begin{align*}
	\|v\,w_n\|_{H_0^1(\Omega)}
	\leq\|w_n\nabla v\|_{L^2(\Omega)}+\|v\nabla w_n\|_{L^2(\Omega)}
	\end{align*}
	holds.
	The second term tends to zero since $v \in L^\infty(\Omega)$ and $\nabla w_n \to 0$ in $L^2(\Omega)$.
	Moreover, for a subsequence, the term $\|w_n\nabla v\|_{L^2(\Omega)}$
	converges to zero aswell.
	To see this, pick any subsequence and choose a subsubsequence, for
	simplicity again denoted by $(w_n)_n$,
	that converges to $0$ pointwise q.e., see \cref{Lem:Capacity}, and thus pointwise $\lambda^d$-a.e.
	Now, $\|w_n\nabla v\|_{L^2(\Omega)}\to 0$
	follows from Lebesgue's dominated convergence theorem and $|w_n
	\nabla v| \leq C |\nabla v| \in L^2(\Omega)$.
	The assertion for the whole sequence is obtained by a
	subsequence-subsequence argument.
\end{proof}

Now, we derive a characterization of the multiplier
as the difference of two nonnegative measures. 
Related results can be found in \cite[Lem.~4.2,~Thm.~4.3]{WachsmuthPointwiseConstraints},
see also \cite{BrowderBrezis,GrunRehomme}.
Therein, the admissible set does not depend on the spatial variable 
and, thus, the results do not apply to our setting immediately.

\begin{theorem}
	\label{Lem:SignedMeasure}
	Assume that $\psi, \varphi$ fulfill the conditions of \cref{AssumptionObstacles}.
	Let $\zeta \in H^{-1}(\Omega)$ be arbitrary and
	set $y:=S_{\operatorname{id}}(\zeta)$, $\xi:=Ly-\zeta$.
	Then the following statements hold.
	\begin{enumerate}
		\item 
		\label[theorem]{C5:Lem:SignedMeasure1}
		$\xi \in H^{-1}(\Omega)$ 
		acts as the difference $\tilde{\xi}_\psi-\tilde{\xi}^\varphi$ of nonnegative measures
		$\tilde{\xi}_\psi, \tilde{\xi}^\varphi\in \mathcal{M}_+(\Omega)$
		on all elements of $H_0^1(\Omega) \cap C_{\mathrm{c}}(\Omega)$, i.e.,
		\begin{align}
		\label{Characterization}
		\langle \xi,w\rangle
		=\int_{\Omega} w\,\diff\tilde{\xi}_\psi-\int_\Omega w\,\diff\tilde{\xi}^\varphi
		\end{align}
		holds for all $w \in H_0^1(\Omega) \cap C_{\mathrm{c}}(\Omega)$.
		
		\item
		\label[theorem]{C5:Lem:SignedMeasure2}
		Let $A \subseteq \Omega$
		be an arbitrary Borel set.
		Then $\operatorname{cap}(A)=0$
		implies $\tilde{\xi}_\psi(A)=\tilde{\xi}^\varphi(A)=0$.
		
		\item
		\label[theorem]{C5:Lem:SignedMeasure3}
		The characterization \cref{Characterization}
		carries over to all $w \in H_0^1(\Omega) \cap L^\infty(\Omega)$.
		In particular, 
		the quasi-continuous 
		and Borel measurable
		representatives of $w$ are $\tilde{\xi}_\psi$- and $\tilde{\xi}^\varphi$-integrable.
		
		\item
		\label[theorem]{C5:Lem:SignedMeasure4}
		Furthermore,
		it holds $y=\psi$ $\tilde{\xi}_\psi$-a.e.\ on $\Omega$
		and $y=\varphi$ $\tilde{\xi}^\varphi$-a.e.\ on $\Omega$,
		i.e.,
		$\tilde{\xi}_\psi(I_\psi(\zeta))=0$
		and
		$\tilde{\xi}^\varphi(I^\varphi(\zeta))=0$.

		\item 
		\label[theorem]{C5:Lem:SignedMeasure5}
		Assume $w \in H_0^1(\Omega)\cap L^1(\tilde{\xi}_\psi)$.
		Then we have $w \in L^1(\tilde{\xi}^\varphi)$
		and \cref{Characterization} holds for $w$.
		The opposite statement with exchanged roles of $\tilde{\xi}_\psi$ and $\tilde{\xi}^\varphi$ is also true.
	\end{enumerate}
\end{theorem}

\begin{proof}
	%	\setenumerate{fullwidth}
	%	\begin{enumerate}
	%		\item 
	ad 1.:
	We define
	\begin{align*}
	v:=\frac{y-\psi}{\varphi-\psi}.
	\end{align*}
	By the assumptions on $\psi$ and $\varphi$,
	we have $0 \leq v \leq 1$ and $v \in H^1(\Omega) \cap L^\infty(\Omega)$.
	
	Now, 
	for $w \in H_0^1(\Omega) \cap L^\infty(\Omega)$,
	we have $v\,w$, $(1-v)\, w \in H_0^1(\Omega)$
	%		see \cref{C1:Lem:ProductRuleInH01},
	and
	we write
	\begin{align*}
	\langle \xi,w\rangle
	&=\langle \xi,(1-v)\,w\rangle
	+\langle \xi,v\, w\rangle.
	\end{align*}
	Thus,
	we deduce
	\begin{align*}
	\langle \xi,w\rangle
	=\xi_\psi(w) -\xi^\varphi(w),
	\end{align*}
	where $\xi_\psi$, $\xi^\varphi$ are defined by
	\begin{align*}
	\xi_\psi\colon w \mapsto \langle \xi,(1-v)\, w\rangle,
	\quad
	\xi^\varphi \colon w \mapsto 
	\langle \xi,-v\, w\rangle.
	\end{align*}
	Note that $\xi_\psi$, $\xi^\varphi$
	are nonnegative linear forms on $H_0^1(\Omega)\cap L^\infty(\Omega)$. 
	To see this,
	assume $w \in H_0^1(\Omega)_+ \cap L^\infty(\Omega)$
	and let first $\|w\|_{L^\infty}\leq c_{\psi}^{\varphi}$.
	By definition of $v$,
	we have $-v\,w+y \in K_\psi^\varphi$
	and therefore
	\begin{align}
	\label{Nonnegativity}
	\xi^\varphi(w)
	=\langle \xi,-v\,w+y-y\rangle
	\geq 0.
	\end{align}
	Since $\xi^\varphi$ is linear,
	\cref{Nonnegativity}
	holds for all $w \in H_0^1(\Omega)_+ \cap L^\infty(\Omega)$.
	In a similar fashion,
	we can show that $\xi_\psi$ is nonnegative on 
	$H_0^1(\Omega)_+ \cap L^\infty(\Omega)$.
	
	In particular, $\xi_\psi$, $\xi^\varphi$
	are nonnegative linear forms on $H_0^1(\Omega)\cap C_{\mathrm{c}}(\Omega)$. 
	By \cite[Lem.~6.53]{BonnansShapiro},
	$\xi_\psi$ and $\xi^\varphi$
	have unique nonnegative continuous extensions
	over $C_{\mathrm{c}}(\Omega)$,
	also denoted by $\xi_\psi$, respectively $\xi^\varphi$.
	Moreover,
	by \cite[Thm.~6.54]{BonnansShapiro},
	there are unique nonnegative, regular, locally finite Borel measures $\tilde{\xi}_\psi$, $\tilde{\xi}^\varphi$
	such that
	\begin{align*}
	\xi_\psi(w)=\int_\Omega w\,\diff\tilde{\xi}_\psi
	\quad
	\text{ and }
	\quad
	\xi^\varphi(w)=\int_\Omega w\,\diff\tilde{\xi}^\varphi
	\end{align*}
	holds for all $w \in C_{\mathrm{c}}(\Omega)$.\\

	ad 2.:
	%\item
	%	ad 2.:
	Now, we modify the proof of \cite[Lem.~6.55]{BonnansShapiro}
	to show that for a Borel set $A \subseteq \Omega$, 
	$\operatorname{cap}(A)=0$
	implies $\tilde{\xi}_\psi(A)=\tilde{\xi}^\varphi(A)=0$.
	W.l.o.g.\
	we prove the statement for $\tilde{\xi}^\varphi$.
	
	Let $(\varepsilon_n)_{n \in \mathbb{N}} \subseteq \mathbb{R}_+$ be a sequence with
	$\varepsilon_n \to 0$ as $n \to \infty$.
	Fix $n \in \mathbb{N}$.
	Then we find an open superset $A_n$
	of $A$ in $\Omega$ with 
	$\operatorname{cap}(A_n)<\varepsilon_n$.
	Furthermore, 
	by \cite[Lem.~3.4]{HarderWachsmuth},
	there is $u_n \in H_0^1(\Omega)_+$ 
	satisfying $u_n=1$ q.e.\ on 
	$A_n$
	as well as $\|u_n\|_{H_0^1(\Omega)}^2=\operatorname{cap}(A_n)<\varepsilon_n$.
	Moreover, 
	we can assume $u_n \in L^\infty(\Omega)$ and $0 \leq u_n \leq 1$,
	since $\min(z,1) \in H_0^1(\Omega)$ satisfies $\|\min(z,1)\|_{H_0^1(\Omega)} \leq \|z\|_{H_0^1(\Omega)}$ for $z \in H_0^1(\Omega)$.
	By regularity of $ \tilde{\xi}^\varphi$,
	we can find a compact set $K_n\subseteq A_n$
	satisfying $\tilde{\xi}^\varphi(A_n)\leq\tilde{\xi}^\varphi(K_n)+\varepsilon_n$.
	Using a smooth version of Urysohn's lemma,
	there exists a function $g_n \in H_0^1(\Omega) \cap C_{\mathrm{c}}(\Omega)$ 
	with values in $[0,1]$
	satisfying $g_n=1$ on $K_n$
	and having compact support in $A_n$.
	Then we have $1_{K_n}\leq g_n \leq u_n$ q.e.\ on $\Omega$.
	
	Now, we conclude
	\begin{align*}
	\tilde{\xi}^\varphi(A_n)
	&\leq \tilde{\xi}^\varphi(K_n)+\varepsilon_n
	\leq \int_\Omega g_n \,\diff\tilde{\xi}^\varphi+\varepsilon_n
	=\langle \xi,-v\,g_n\rangle+\varepsilon_n\\
	&\leq \langle \xi,-v\,u_n\rangle+\varepsilon_n
	\leq \|\xi\|_{H^{-1}(\Omega)}\|v\,u_n\|_{H_0^1(\Omega)}+\varepsilon_n.
	\end{align*}
	Using \cref{Lem:ApplicationDominatedConvergence},
	we know that $\|v\,u_n\|_{H_0^1(\Omega)} \to 0$.
	Now,
	$\bigcap_{n \in \mathbb{N}} A_n$ is Borel measurable and
	\begin{align*}
	\tilde{\xi}^\varphi\left(\bigcap_{n \in \mathbb{N}} A_n\right)=0.
	\end{align*}
	Since $A \subseteq \bigcap_{n \in \mathbb{N}} A_n$,
	we conclude $\tilde{\xi}^\varphi(A)=0$.\\
	
	ad 3.:
	%\item
	Now,
	we argue in a similar fashion as in \cite[Lem.~6.56]{BonnansShapiro} 
	to show that each $w \in H_0^1(\Omega) \cap L^\infty(\Omega)$
	satisfies
	\begin{align*}
	\langle \xi,w\rangle=\int_{\Omega} w\,\diff\tilde{\xi}_\psi-\int_{\Omega} w\,\diff\tilde{\xi}^\varphi.
	\end{align*}
	Let $w \in H_0^1(\Omega) \cap L^\infty(\Omega)$.
	Then we find $(\overline{w}_n)_{n \in \mathbb{N}} \subseteq C_{\mathrm{c}}^\infty(\Omega)$ with $\overline{w}_n \to w$ in $H_0^1(\Omega)$.
	Definig
	$w_n:=\max(-\|w\|_{L^\infty(\Omega)},\min(\overline{w}_n,\|w\|_{L^\infty(\Omega)}))$
	we have $w_n \in H^1_0(\Omega) \cap C_{\mathrm{c}}(\Omega)$ and $w_n \to
	w$ in $H^1_0(\Omega)$ as well as $|w_n| \leq \|w\|_{L^\infty(\Omega)}$.
	Then, \cref{Lem:ApplicationDominatedConvergence} yields $v\,w_n \to v\,w$ in $H^1_0(\Omega)$. Therefore, using $\||z|\|_{H_0^1(\Omega)} = \|z\|_{H_0^1(\Omega)}$ for all $z \in H_0^1(\Omega)$, see \cite[Cor.~5.8.1]{AttouchButtazzoMichaille}, 
	with
	\begin{align*}
	\| -v \,| w_n - w_m| \|_{H^1_0(\Omega)} &= \big \| |-v\,|w_n-w_m||\big
	\|_{H^1_0(\Omega)} = \|v\,w_n - v\,w_m\|_{H^1_0(\Omega)}
	%\end{align*}
	\intertext{
		and
	}
	%\begin{align*}
	\|w_n-w_m\|_{L^1(\tilde{\xi}^\varphi)} &= \langle \xi, -v\,|w_n-w_m|\rangle
	\leq \|\xi\|_{H^{-1}(\Omega)} \|-v\,|w_n-w_m| \|_{H^1_0(\Omega)}
	\end{align*}
	we find that $(w_n)_{n \in \mathbb N}$ is a Cauchy sequence in
	$L^1(\tilde{\xi}^\varphi)$. Similarly, $(w_n)_{n \in \mathbb N}$ is a
	Cauchy sequence in $L^1(\tilde{\xi}_\psi)$. Thus, a subsequence converges
	pointwise $\tilde{\xi}^\varphi$-a.e. and $\tilde{\xi}_\psi$-a.e. to an
	element in $L^1(\tilde{\xi}^\varphi) \cap L^1(\tilde{\xi}_\psi)$.
	Now, pick a subsubsequence that converges pointwise q.e. to $w$. Then, by the second statement of
	\cref{Lem:SignedMeasure} and since
	the convergence holds except on a Borel set of capacity zero, $w \in L^1(\tilde{\xi}^\varphi)\cap
	L^1(\tilde{\xi}_\psi)$ and Lebesgue's dominated convergence theorem
	implies
	\begin{align*}
	\langle \xi^\varphi,w\rangle &= \underbrace{\langle
		\xi,-v(w-w_n)\rangle}_{\to 0 \text{ as }n \to \infty} +
	\underbrace{\int_{\Omega} w_n \,\diff\tilde{\xi}^\varphi}_{\to
		\int_{\Omega} w \,\diff\tilde{\xi}^\varphi \text{ along a subsequence}}.
	\end{align*}
	This yields $\langle \xi^\varphi,w\rangle = \int_{\Omega} w \,
	d\tilde{\xi}^\varphi$ and, similarly, $\langle \xi_\psi,w\rangle =
	\int_{\Omega} w \,\diff\tilde{\xi}_\psi$.\\

	ad 4.:
	%		\item
	We modify the proof of
	\cite[Prop.~2.5]{WachsmuthStrongStationarity}.
	We consider a smooth cut-off function
	$\chi \in C_{\mathrm{c}}^\infty(\Omega)$
	with $0 \leq \chi \leq 1$
	and $\chi=1$
	on a compact set $K \subseteq \Omega$.
	
	We define
	$w:=\chi\left[(1-v)\,\psi+v\,y\right]+(1-\chi)\,y$
	and obtain
	$w \in K_\psi^\varphi$.
	This implies
	\begin{align*}
	0 &\leq 
	\langle \xi,w-y\rangle
	=\langle \xi,\chi\,(1-v)\,\psi+\chi\, v\,y-\chi\, y\rangle\\
	&=\langle \xi,(1-v)\,\chi\,(\psi-y)\rangle
	=\int_\Omega \chi\,(\psi-y)\,\diff\tilde{\xi}_\psi.
	\end{align*}
	Since $\chi\,(\psi-y)\leq 0$ q.e.\ on $\Omega$,
	and thus, by the second statement of the theorem, $\tilde{\xi}_\psi$-a.e.,
	we conclude $y=\psi$ $\tilde{\xi}_\psi$-a.e.\ on $K$.
	Covering $\Omega$ with countably many compact subsets,
	we infer
	$y-\psi=0$ $\tilde{\xi}_\psi$-a.e.\ on $\Omega$.
	Similarly, we can show $\varphi-y=0$ $\tilde{\xi}^\varphi$-a.e.\ on $\Omega$.\\

	ad 5.:
	%		\item
	Assume $w \in H_0^1(\Omega) \cap L^1(\tilde{\xi}_\psi)$.
	We approximate $w$ in $H_0^1(\Omega)$ by $(w_n)_{n \in \mathbb{N}}$
	defined via $w_n:=\max(-n,\min(n,w))$.
	Then we have $w_n \to w$ in $H_0^1(\Omega)$
	and $w_n \to w$ pointwise $\tilde{\xi}_\psi$-a.e.\
	(after choosing a subsequence).
	Since $|w_n|\leq |w|$ and since $w \in L^1(\tilde{\xi}_\psi)$,
	we apply Lebesgue's dominated convergence theorem
	and obtain $w_n \to w$ in $L^1(\tilde{\xi}_\psi)$.

	From
	\begin{align*}
	\|w_n-w_m\|_{L^1(\tilde{\xi}^\varphi)}
	&=\int_\Omega |w_n-w_m|\,\diff\tilde{\xi}^\varphi
	=\int_\Omega |w_n-w_m|\,\diff\tilde{\xi}_\psi-\langle \xi,|w_n-w_m|\rangle\\
	&\leq \|w_n-w_m\|_{L^1(\tilde{\xi}_\psi)}+\|\xi\|_{H^{-1}(\Omega)}\|w_n-w_m\|_{H_0^1(\Omega)}
	\end{align*}
	it follows that $(w_n)_{n \in \mathbb{N}}$ is a Cauchy sequence in $L^1(\tilde{\xi}^\varphi)$
	and we can again conclude that $w_n \to w$ in $L^1(\tilde{\xi}^\varphi)$.
	
	From the representation
	\begin{align*}
	\langle \xi,w_n\rangle=\int_\Omega w_n \,\diff\tilde{\xi}_\psi-\int_\Omega w_n \,\diff\tilde{\xi}^\varphi
	\end{align*}
	for all $n \in \mathbb{N}$,
	since $w_n \to w$ in $H_0^1(\Omega)$, $L^1(\tilde{\xi}_\psi)$ and $L^1(\tilde{\xi}^\varphi)$,
	we conclude
	\begin{align*}
	\langle \xi,w\rangle=\int_\Omega w \,\diff\tilde{\xi}_\psi-\int_\Omega w\,\diff\tilde{\xi}^\varphi.
	\end{align*}
	The opposite statement follows similarly.
	%	\end{enumerate}
\end{proof}

In the subsequent lemma
we assume that the active sets 
$A_\psi(\zeta)$ and $A^\varphi(\zeta)$
have a positive distance.
With this condition
we mean that
there are representative quasi-closed sets Borel measurable
$\tilde{A}_\psi(\zeta)$, $\tilde{A}^\varphi(\zeta)$
such that $\operatorname{dist}(\tilde{A}_\psi(\zeta),\tilde{A}^\varphi(\zeta))>0$
and the two sets coincide with $A_\psi(\zeta), A^\varphi(\zeta)$
up to a set of capacity zero.

\begin{lemma}
	\label{Lem:DistanceActiveSets}
	Assume that $\psi, \varphi$ fulfill the conditions of \cref{AssumptionObstacles}. 
	Let $\zeta \in H^{-1}(\Omega)$ be arbitrary and set
	$y:=S_{\operatorname{id}}(\zeta)$ and $\xi:=Ly-\zeta$. 
	Suppose $A_\psi(\zeta)$ and $A^\varphi(\zeta)$ have a positive distance.
	Then, with $\tilde{\xi}_\psi, \tilde{\xi}^\varphi\in \mathcal{M}_+(\Omega)$ as in \cref{Lem:SignedMeasure},
	it holds $H_0^1(\Omega) \subseteq L^1(\tilde{\xi}_\psi) \cap L^1(\tilde{\xi}^\varphi)$ and
	\begin{align*}
	\langle \xi,w\rangle=\int_\Omega w\,\diff\tilde{\xi}_\psi-\int_\Omega w \,\diff\tilde{\xi}^\varphi
	\end{align*}
	for all $w \in H_0^1(\Omega)$.
\end{lemma}

\begin{proof}
	Since the active sets have a positive distance $C>0$,
	we can find an element $v_2$
	satisfying $v_2 \in C^\infty(\mathbb{R}^d)$
	as well as $v_2=1$ on $A_\psi(\zeta)$ and $v_2=0$
	outside $A_\psi(\zeta)+B_{C/2}$.
	
	Since $v_2$ is smooth,
	we have $v_2\,w, (v_2-1)\,w \in H_0^1(\Omega)$ for all $w \in H_0^1(\Omega)$
	and
	we define the functionals
	\begin{align*}
	\xi_\psi^{2} \colon w \mapsto \langle \xi, v_2\,w\rangle,
	\quad
	\xi^\varphi_{2} \colon w \mapsto \langle \xi, (v_2-1)\,w\rangle
	\end{align*}
	on $H_0^1(\Omega)$.
	Since $v_2$ is in $C^\infty(\mathbb{R}^d)$,
	it is easy to show that $\xi_\psi^{2}$ and $\xi^\varphi_{2}$
	are bounded linear functionals on $H_0^1(\Omega)$.
	Moreover, we have
	\begin{align*}
	\langle \xi,w\rangle
	=\langle \xi_\psi^{2},w\rangle-\langle \xi^\varphi_{2},w\rangle
	\end{align*}
	for all $w \in H_0^1(\Omega)$.
	
	Assume first $w$ is in $H_0^1(\Omega) \cap L^\infty(\Omega)$.
	Then we have
	\begin{align*}
	\int_{\Omega} w \,\diff\tilde{\xi}_\psi
	&=\int_{\Omega} v_2\, w \,\diff\tilde{\xi}_\psi
	=\int_{\Omega} v_2 \,w \,\diff\tilde{\xi}_\psi-\int_{\Omega} v_2 \,w \,\diff\tilde{\xi}^\varphi
	=\langle \xi,v_2\, w\rangle
	=\langle\xi_\psi^{2},w\rangle.
	\end{align*}
	Here,
	the first equation holds since $\tilde{\xi}_\psi(I_\psi(\zeta))=0$
	and $w=v_2\,w$ q.e.\ and thus $\tilde{\xi}_\psi$-a.e.\ on $A_\psi(\zeta)$,
	see \cref{Lem:SignedMeasure}.
	Similarly, the second equation holds
	since $v_2\,w=0$ $\tilde{\xi}^\varphi$-a.e.\ on $\Omega$.
	
	Let $w \in H_0^1(\Omega)$.
	Now we have $\max(-n,\min(w,n)) \in H_0^1(\Omega) \cap L^\infty(\Omega)$
	and $w_n \to w$ in $H_0^1(\Omega)$.
	Furthermore,
	\begin{align*}
	\|w_n-w_m\|_{L^1(\tilde{\xi}_\psi)}
	=\int_{\Omega} |w_n-w_m|\,\diff\tilde{\xi}_\psi
	=\langle \xi_\psi^{2},|w_n-w_m|\rangle
	\leq \|\xi_\psi^{2}\|_{H^{-1}}\||w_n-w_m|\|_{H_0^1}.
	\end{align*}
	Thus,
	$(w_n)_{n \in \mathbb{N}}$ is a Cauchy sequence in $L^1(\tilde{\xi}_\psi)$.
	Since $w_n \to w$ pointwise q.e.\ and thus $\tilde{\xi}_\psi$-a.e.,
	see \cref{Lem:SignedMeasure},
	we have
	$w_n \to w$ in $L^1(\tilde{\xi}_\psi)$.
	Arguing for $\tilde{\xi}^\varphi$ in a similar fashion,
	we obtain that
	\begin{align*}
	\langle \xi,w\rangle
	=\langle \xi_\psi^{2},w\rangle-\langle \xi^\varphi_{2},w\rangle
	=\int_{\Omega} w \,\diff\tilde{\xi}_\psi-\int_{\Omega} w \,\diff\tilde{\xi}^\varphi.
	\end{align*}
\end{proof}

The following example shows that, in general,
i.e., when the active sets do not have a positive distance,
the characterization of the functional $LS_{\operatorname{id}}(\zeta)-\zeta$ as 
the difference of the two measures $\tilde{\xi}_\psi$ and $\tilde{\xi}^\varphi$
does not need to apply for all possible arguments in $H_0^1(\Omega)$.
A related example 
which is
independent from the connection to the bilateral obstacle problem
can be found in \cite[App.~2]{WachsmuthPointwiseConstraints}.

\begin{example}
	\label{ex:counterexample}
	For $d=2$ and for $0<\beta<\tfrac{1}{2}$, consider the function
	\begin{equation}\label{ydef}
	y(x)=\begin{cases} \sin((-\ln(|x|))^{\beta}), \quad &x\in \Omega:=B_{\rho}(0) \setminus \{0\}\\ 0, \quad &x=0
	\end{cases},\quad
	\rho=\exp(-\pi^{1/\beta})<1.
	\end{equation}
	Then $y|_{\partial B_{\rho}(0)}=\sin(\pi)=0$ and $y\in
	H_0^1(\Omega)\cap
	L^\infty(\Omega)$, since $|y|\leq 1$ as well as
	\begin{align*}
	|\nabla y(x)|^2 
	&=\frac{\cos^2((- \ln(|x|))^\beta) \beta^2
		(-\ln(|x|))^{2\beta-2}
	}{
		|x|^2} \frac{x_1^2+x_2^2}{|x|^2}
	\le
	\beta^2 (-\ln(|x|))^{2\beta-2} |x|^{-2}
	%		\end{align*}
	\intertext{and thus}
	%		\begin{align*}
	\|y\|_{H^1_0(\Omega)}^2&= \|\nabla y\|_{L^2(\Omega)}^2
	\le 2 \pi \beta^2 \int_0^{\rho}
	(-\ln(r))^{2\beta-2}  r^{-2} r \,\diff r
	=
	2 \pi \beta^2 \frac{(-\ln(r))^{2\beta-1}}{1-2\beta} \Big |_{0}^\rho
	<\infty,
	\end{align*}
	since $\beta<\tfrac 12$.
	Now, we consider the obstacles given by $\psi(x):
	=\min\left(-\tfrac{1}{2},y(x)\right)$,
	$\varphi(x):=\max\left(\tfrac{1}{2},y(x)\right)$.
	
	We have
	\begin{align*}
	(-\ln(r(t)))^\beta=t\quad
	\Leftrightarrow \quad r(t)=\exp(-t^{1/\beta})
	\end{align*}
	and, for $k \in \mathbb{N}$, we set
	$r_k^\pm:=r(2k\pi\pm \pi/2)$.
	This choice implies
	$\rho=r(\pi)>r_1^->r_1^+>r_2^->r_2^+>\ldots>0$
	and
	$y(r_k^\pm (\cos t,\sin t))=\pm 1$
	for all $t \in (0,2\pi)$.

	Now, let $\omega_k>0$ be weights (that will be adjusted below), with
	$\sum_{k=1}^\infty \omega_k^2<\infty$
	and consider the functional
	\begin{align}
	\label{Def:Functional}
	\langle \xi,w\rangle
	:=\sum_{k=1}^\infty \frac{\omega_k}{\sqrt{\ln(r_k^-/r_k^+)}}
	\int_0^{2\pi} (w(r_k^- (\cos t,\sin t))-w(r_k^+ (\cos t,\sin
	t)))\,\diff t
	\end{align}
	for $w\in H_0^1(\Omega)$.
	Note that the integral in \cref{Def:Functional}
	is well-defined.
	To see this, observe first
	that the quasi-continuous Borel measurable representatives of $w$ are unique up to a
	set of capacity zero.
	Let now $E \subseteq \Omega$ be a Borel set of capacity zero.
	Then,
	for any radius $0<R<\rho$,
	by \cite[Thm.~7.5]{Helms},
	the surface measure $\sigma_R$ on the sphere $\partial B_R$
	satisfies $\sigma_R(E\cap \partial B_R)=0$. 
	
	We have $\xi\in H^{-1}(\Omega)$, since
	\begin{align*}
	&|\langle \xi,w\rangle|\le \sum_{k=1}^\infty
	\frac{\omega_k}{\sqrt{\ln(r_k^-/r_k^+)}}
	\int_0^{2\pi} \int_{r_k^+}^{r_k^-} |\nabla w(s (\cos t,\sin t))|
	\sqrt{s} \frac{1}{\sqrt{s}} \,\diff s\,\diff t\\
	&\le\sum_{k=1}^\infty \frac{\omega_k}{\sqrt{\ln(r_k^-/r_k^+)}}
	\|\nabla w\|_{L^2(B_{r_k^-}(0)\setminus B_{r_k^+}(0))}
	\left(\int_0^{2\pi} \int_{r_k^+}^{r_k^-} \frac{1}{s} \,\diff
	s\,\diff t\right)^{\frac{1}{2}}\\
	&=\sqrt{2\pi}\sum_{k=1}^\infty \omega_k \|\nabla
	w\|_{L^2(B_{r_k^-}(0)\setminus
		B_{r_k^+}(0))}\\
	&\le \sqrt{2\pi} \left(\sum_{k=1}^\infty \omega_k^2
	\right)^{\frac{1}{2}}
	\left(\sum_{k=1}^\infty \|\nabla w\|_{L^2(B_{r_k^-}(0)\setminus
		B_{r_k^+}(0))}^2 \right)^{\frac{1}{2}}\\
	&\le \sqrt{2\pi} \left(\sum_{k=1}^\infty \omega_k^2
	\right)^{\frac{1}{2}} \|\nabla w\|_{L^2(\Omega)}.
	\end{align*}
	Here, we have used the Cauchy Schwarz inequality and the fact
	that, for all $k \in \mathbb{N}$, the sets
	$B_{r_k^-}(0)\setminus
	B_{r_k^+}(0)$ are disjoint subsets of $\Omega$.
	
	For
	$w\in H_0^1(\Omega)\cap L^\infty(\Omega)$, we have
	\begin{align*}
	%\begin{split}
	\langle \xi,w\rangle&=\int_\Omega w
	\,\diff\tilde{\xi}_\psi-\int_\Omega w \,\diff\tilde{\xi}^\varphi,\\
	&=\int_{\Omega} w \,\diff \left(\sum_{k=1}^\infty
	\frac{\omega_k}{\sqrt{\ln(r_k^-/r_k^+)}r_k^-}\sigma_{r_k^-}\right)
	-\int_{\Omega} w \,\diff \left(\sum_{k=1}^\infty
	\frac{\omega_k}{\sqrt{\ln(r_k^-/r_k^+)}r_k^+}\sigma_{r_k^+}\right)
	%\end{split}
	\end{align*}
	where $\tilde{\xi}_\psi, \tilde{\xi}^{\varphi}$ are nonnegative
	finite measures with support in $\{y=\psi\}$ and $\{y=\varphi\}$, respectively.
	In fact, to show that $\tilde{\xi}^\varphi$ is a finite measure, we observe
	that
	\begin{align}\label{C5:lnrest}
	\ln(r_k^-/r_k^+)=%\ln(r_k^-)-\ln(r_k^+)=
	(2k\pi+\pi/2)^{1/\beta}-(2k\pi-\pi/2)^{1/\beta}
	\begin{cases} \ge (2k\pi-\pi/2)^{1/\beta-1} \pi/\beta,\\
	\le (2k\pi+\pi/2)^{1/\beta-1} \pi/\beta.
	\end{cases}
	\end{align}
	Hence, for any $w\in C(\overline\Omega)$ with $0\le w\le 1$
	\begin{align*}
	&0\le \int_\Omega w \,\diff\tilde{\xi}^\varphi\le
	\sum_{k=1}^\infty \frac{\omega_k}{\sqrt{\ln(r_k^-/r_k^+)}}
	\int_0^{2\pi} 1 \,\diff t
	\le 2\pi \left(\sum_{k=1}^\infty \omega_k^2 \right)^{\frac{1}{2}}
	\left(\sum_{k=1}^\infty \frac{1}{\ln(r_k^-/r_k^+)}
	\right)^{\frac{1}{2}}\\
	&\le 2\pi \left(\sum_{k=1}^\infty \omega_k^2 \right)^{\frac{1}{2}}
	\left(\sum_{k=1}^\infty \frac{1}{\pi/\beta
		(2k\pi-\pi/2)^{1/\beta-1}}\right)^{\frac{1}{2}}
	\le C
	\end{align*}
	with a constant $C>0$, since $\beta <1/2$.

	Set $\zeta:=Ly-\xi$.
	We argue that the function $y \in K_\psi^\varphi$ as defined in
	\cref{ydef} satisfies
	the bilateral obstacle problem \cref{BOP} for $f=\operatorname{id}$,
	i.e., $y=S_{\operatorname{id}}(\zeta)$. First note that $y \in
	K_\psi^\varphi \subseteq H^1_0(\Omega)\cap L^\infty(\Omega)$ by the choice
	of $y,\psi$ and $\varphi$. Now, let $z \in K_\psi^\varphi$ be arbitrary.
	Then $z -y \geq 0$ q.e.\ on $\{y=\psi\}$ and $z -y\leq 0$ q.e.\ on $\{y=\varphi\}$.
	Additionally, we have
	\begin{align}
	\label{C5:suppinclusion1}
	\operatorname{supp}(\tilde \xi_\psi) \setminus \{0\} &= \bigcup_{k=1}^\infty \partial B_{r_k^-}
	\subseteq \{y=\psi\}
	\intertext{and}
	\label{C5:suppinclusion2}
	\operatorname{supp}(\tilde \xi^\varphi) \setminus \{0\}&=\bigcup_{k =1}^\infty \partial B_{r_k^+}
	\subseteq \{y=\varphi\}.
	\end{align}
	This yields
	\begin{align*}
	\langle\xi,z-y\rangle \,%&\stackrel{\cref{Characterization}}{=} \,
	&=\int_{\Omega}
	(z-y) \,\diff \tilde{\xi}_\psi - \int_{\Omega} (z-y) \,\diff
	\tilde{\xi}^\varphi\\
	&= \, \int_{\{y=\psi\}} (z-y) \,\diff \tilde{\xi}_\psi -
	\int_{{\{y=\varphi\}}} (z-y) \,\diff \tilde{\xi}^\varphi \geq 0
	\end{align*}
	and we obtain $y=S_{\operatorname{id}}(\zeta)$.
	
	Note that by \cref{C5:suppinclusion1} and \cref{C5:suppinclusion2} we find
	\begin{align*}
	\operatorname{dist}(A_\psi(\zeta),A^\varphi(\zeta))=\operatorname{dist}(\{y=\psi\},\{y=\varphi\})=0
	\end{align*}
	since $r_k^\pm \searrow 0$. Thus,
	\cref{Lem:DistanceActiveSets}
	does not apply.

	Now, consider the unbounded function
	$w(x)=(-\ln(|x|))^{\beta}-\pi$. Then since
	\begin{align*}
	|\nabla w(x)|^2 = \frac{\beta^2 (-\ln(|x|))^{2\beta-2}}{|x|^2}
	\frac{x_1^2+x_2^2}{|x|^2}
	=\beta^2 (-\ln(|x|))^{2\beta-2} |x|^{-2}
	,
	\end{align*}
	we have
	$w\in H_0^1(\Omega)$ as above.
	With $\omega_k=k^{-1}$, $\beta = \tfrac 13$
	and by using \cref{C5:lnrest},
	we obtain the estimate
	\begin{align*}
	\int_\Omega w \,\diff\tilde \xi^\varphi &=
	\sum_{k=1}^\infty \frac{\omega_k}{\sqrt{\ln(r_k^-/r_k^+)}}
	\int_0^{2\pi} (2k\pi+\tfrac{\pi}{2}-\pi)\,\diff t\\
	&\ge
	\sum_{k=1}^\infty \frac{2\pi
		(2k\pi-\tfrac{\pi}{2})}{k(2k\pi+\tfrac{\pi}{2})\sqrt{3\pi}}
	\ge
	\sum_{k=1}^\infty \frac{2\pi \, \tfrac 32 k \pi}{k \, \tfrac 52 k \pi \,
		\sqrt{3\pi}}=\infty.
	\intertext{
		Similarly, we find
	}
	\int_{\Omega} w \,\diff \tilde{\xi}_\psi &\geq \sum_{k=1}^\infty
	\frac{2\pi \, \tfrac 12 k \pi}{k \, \tfrac 52 k \pi \, \sqrt{3\pi}} =
	\infty.
	\end{align*}
	This shows that, in general, the representation
	\cref{Characterization}
	does not hold for all $w \in H^1_0(\Omega)$.
\end{example}

In the following lemma,
we find a characterization of the critical cone.
In parts,
the proof is based on the proof of
\cite[Lem.~3.1]{WachsmuthStrongStationarity}. 

\begin{lemma}
	\label{Lem:CriticalCone}
	Assume that $\psi, \varphi$ fulfill the conditions of \cref{AssumptionObstacles}.
	Let $\zeta \in H^{-1}(\Omega)$ be arbitrary and
	set $y:=S_{\operatorname{id}}(\zeta)$, $\xi:=Ly-\zeta$.
	Then the critical cone
	\begin{align*}
	\mathcal{K}_{K_\psi^\varphi}(y,\xi):=T_{K_\psi^\varphi}(y)\cap \xi^\perp
	\end{align*}	
	has the following structure. 
	There exist quasi-closed sets
	$A_\psi^{\mathrm{s}}(\zeta) \subseteq A_\psi(\zeta)$
	and
	$A_{\mathrm{s}}^\varphi(\zeta) \subseteq A^\varphi(\zeta)$
	which are unique up to sets of capacity zero
	such that
	\begin{align}
	\label{StructureCriticalCone}
	\begin{split}
	&\mathcal{K}_{K_\psi^\varphi}(y,\xi)
	=\{z \in H_0^1(\Omega) \mid z \geq 0 \text{ q.e.\ in } A_\psi(\zeta), z \leq 0 \text{ q.e.\ in } A^\varphi(\zeta) \text{ and } \langle \xi,z\rangle=0\}\\
	&=\left\{z \in H_0^1(\Omega) \mid z \geq 0\text{ q.e.\ in } A_\psi(\zeta), z \leq 0 \text{ q.e.\ in } A^\varphi(\zeta),
	z=0 ~\tilde{\xi}_\psi \text{- and } \tilde{\xi}^\varphi
	\text{-a.e.}\right\}\\
	&=\left\{z \in H_0^1(\Omega)\mid z \geq 0 \text{ q.e.\ in } A_\psi(\zeta), z \leq 0 \text{ q.e.\ in } A^\varphi(\zeta),\right.\\
	&\qquad\qquad\qquad\qquad\qquad\qquad\qquad\quad~~
	\left. z=0 \text{ q.e.\ in }A_\psi^{\mathrm{s}}(\zeta)\cup
	A_{\mathrm{s}}^\varphi(\zeta)\right\},
	\end{split}
	\end{align}
	where $\tilde{\xi}_\psi, \tilde{\xi}^\varphi$ denote the
	decomposition of $\xi$ according to \cref{Lem:SignedMeasure}.
\end{lemma}

\begin{proof}
	Recalling \cref{Tan},
	we see that the first equation in \cref{StructureCriticalCone} holds.
	
	We show the second identity in \cref{StructureCriticalCone}.
	Assume $z$ is an element of $\mathcal{K}_{K_\psi^\varphi}(y,\xi)$.
	By polyhedricity,
	see e.g.\ \cite[Thm.~3.2]{Mignot},
	we have
	\begin{align*}
	T_{K_\psi^\varphi}(y)\cap \xi^\perp
	=\overline{R_{K_\psi^\varphi}(y)\cap \xi^\perp},
	\end{align*}
	where $R_{K_\psi^\varphi}(y):=\{z \in H_0^1(\Omega)\mid \exists t>0, y+tz \in K_\psi^\varphi\}$ denotes the radial cone.
	Thus,
	there is a sequence $(z_n)_{n \in \mathbb{N}} \subseteq R_{K_\psi^\varphi}(y)\cap\xi^\perp$
	with $z_n \to z$ in $H_0^1(\Omega)$.
	Since $\psi, \varphi$ are elements of $L^\infty(\Omega)$,
	and by the structure of $R_{K_\psi^\varphi}(y)$
	we conclude $(z_n)_{n \in \mathbb{N}} \subseteq L^\infty(\Omega)$.
	Using \cref{Lem:SignedMeasure},
	each $z_n$ is integrable with respect to $\tilde{\xi}_\psi$ and $\tilde{\xi}^\varphi$.
	Let $n \in \mathbb{N}$ be fixed.
	By $R_{K_\psi^\varphi}(y)\subseteq T_{K_\psi^\varphi}(y)$
	we have $z_n \geq 0$ q.e.\ in $A_\psi(\zeta)$
	and $z_n \leq 0$ q.e.\ in $A^\varphi(\zeta)$.
	This implies
	\begin{align}
	\label{C5:Hilfsresultat2}
	0=\langle \xi,z_n\rangle=\int_\Omega z_n \,\diff\tilde{\xi}_\psi - \int_{\Omega} z_n \,\diff\tilde{\xi}^\varphi
	=\int_{A_\psi(\zeta)} z_n \,\diff\tilde{\xi}_\psi-\int_{A^\varphi(\zeta)} z_n \,\diff\tilde{\xi}^\varphi,
	\end{align} 
	as $\tilde{\xi}_\psi(I_\psi(\zeta)))=0$ and $\tilde{\xi}^\varphi(I^\varphi(\zeta))=0$,
	see \cref{Lem:SignedMeasure}.
	Since $z_n \geq 0$ $\tilde{\xi}_\psi$-a.e.\ on $A_\psi(\zeta)$ and
	$z_n \leq 0$ $\tilde{\xi}^\varphi$-a.e.\ on $A^\varphi(\zeta)$, cf.\ \cref{Lem:SignedMeasure},
	we conclude that 
	$z_n=0$ $\tilde{\xi}_\psi$-a.e.\ on $A_\psi(\zeta)$ and $z_n=0$ $\tilde{\xi}^\varphi$-a.e.\ on $A^\varphi(\zeta)$.
	Using once more that $\tilde{\xi}_\psi(I_\psi(\zeta))=0$ and $\tilde{\xi}^\varphi(I^\varphi(\zeta))=0$,
	we can see that this means $z_n=0$ $\tilde{\xi}_\psi$- and $\tilde{\xi}^\varphi$-a.e.\ on $\Omega$.
	Since $z_n \to z$ for a subsequence pointwise q.e.\ and thus $\tilde{\xi}_\psi$-
	and $\tilde{\xi}^\varphi$-a.e.,
	see \cref{Lem:SignedMeasure},
	we conclude $z=0$ $\tilde{\xi}_\psi$- and $\tilde{\xi}^\varphi$-a.e.
	
	Vice versa, assume
	$z \in T_{K_\psi^\varphi}(y)$ and $z=0$ $\tilde{\xi}_\psi$- and $\tilde{\xi}^\varphi$-a.e.
	Using \cref{Lem:SignedMeasure},
	we see that \cref{C5:Hilfsresultat2} holds (with $z_n$ replaced by $z$) and thus $z \in T_{K_\psi^\varphi}(y)\cap \xi^\perp$ follows.
	
	Thus, we have shown the second equation in \cref{StructureCriticalCone}.
	
	To show the last identity in \cref{StructureCriticalCone} we note
	that by \cite[Thm.~1]{Stollmann},
	there exist quasi-closed sets
	$A_\psi^{\mathrm{s}}(\zeta)$ and $A_{\mathrm{s}}^\varphi(\zeta)$ such that
	\begin{align}
	\label{StrictlyActiveSet}
	\{z \in H_0^1(\Omega) \mid z=0 ~\tilde{\xi}_\psi\text{-a.e.}\}=\{z \in H_0^1(\Omega) \mid z=0 \text{ q.e.\ on } A_\psi^{\mathrm{s}}(\zeta)\}
	\end{align}
	and 
	\begin{align}
	\label{StrictlyActiveSet2}
	\{z \in H_0^1(\Omega) \mid z=0 ~\tilde{\xi}^\varphi\text{-a.e.}\}=\{z \in H_0^1(\Omega)\mid z=0 \text{ q.e.\ on } A_{\mathrm{s}}^\varphi(\zeta)\}.
	\end{align}
	We have $y-\psi=0$ $\tilde{\xi}_\psi$-a.e.\ 
	and thus $y-\psi=0$ q.e.\ on $A_\psi^{\mathrm{s}}(\zeta)$,
	see \cref{StrictlyActiveSet},
	which implies
	$A_\psi^{\mathrm{s}}(\zeta) \subseteq A_\psi(\zeta)$ up to a set of capacity zero.
	The same arguments apply to show $A_{\mathrm{s}}^\varphi(\zeta) \subseteq A^\varphi(\zeta)$ up to a set of capacity zero.
\end{proof}

\begin{corollary}
	\label{Lem:0OnStrictlyActiveSet}
	Assume that $\psi, \varphi$ fulfill the conditions of \cref{AssumptionObstacles}.
	Let $\zeta \in H^{-1}(\Omega)$ be arbitrary and
	set $y:=S_{\operatorname{id}}(\zeta)$, $\xi:=Ly-\zeta$.
	For $w \in H_0^1(\Omega)$ it holds
	$w=0$ q.e.\ on $A_\psi^{\mathrm{s}}(\zeta)$,
	respectively $w=0$ q.e.\ on $A_{\mathrm{s}}^\varphi(\zeta)$
	if and only if
	$w=0$ $\tilde{\xi}_\psi$-a.e.,
	respectively $w=0$ $\tilde{\xi}^\varphi$-a.e.
	In both cases we have
	$w \in L^1(\tilde{\xi}_\psi) \cap L^1(\tilde{\xi}^\varphi)$
	and $\langle \xi,w\rangle=\int_\Omega w\,\diff\tilde{\xi}_\psi-\int_\Omega w \,\diff\tilde{\xi}^\varphi$.
\end{corollary}

\begin{proof}
	The equivalence is implied by the proof of
	\cref{Lem:CriticalCone},
	see \cref{StrictlyActiveSet} and \cref{StrictlyActiveSet2}.
	The statements $w \in L^1(\tilde{\xi}_\psi) \cap L^1(\tilde{\xi}^\varphi)$ 
	and $\xi,w\rangle=\int_\Omega w\,\diff\tilde{\xi}_\psi-\int_\Omega w \,\diff\tilde{\xi}^\varphi$
	follow from \cref{Lem:SignedMeasure}.
\end{proof}

In the following sections,
we also write 
$A_{\mathrm{s}}(\zeta):=A_\psi^{\mathrm{s}}(\zeta)\cup A_{\mathrm{s}}^\varphi(\zeta)$
for the strictly active set
with respect to both obstacles,
we have $A_{\mathrm{s}}(\zeta)\subseteq A(\zeta)$.

Moreover,
we will use the notation $A_\psi^{\mathrm{w}}(\zeta):=A_\psi(\zeta) \setminus A_\psi^{\mathrm{s}}(\zeta)$ for the weakly active set with respect to the lower obstacle $\psi$
and $A_{\mathrm{w}}^\varphi(\zeta):=A^\varphi(\zeta) \setminus A_{\mathrm{s}}^\varphi(\zeta)$ 
for the weakly active set with respect to the upper obstacle $\varphi$.
For the sake of completeness,
we also introduce the notation
$A_{\mathrm{w}}(\zeta):=A_\psi^{\mathrm{w}}(\zeta)\cup A_{\mathrm{w}}^\varphi(\zeta)$
for the weakly active set with respect to upper and lower obstacle.

\subsubsection{G\^ateaux Differentiability of the Solution Operator}

As in the case of 
unilateral obstacle problems,
in points $u$ at which $S_f$ is G\^ateaux differentiable,
we can replace the critical cone in the characterization of the 
directional derivative by the largest linear subset contained in the critical cone, 
and by the linear hull of the critical cone,
respectively.
Both versions
yield a characterization of the G\^ateaux derivative.
The reasoning for these facts in the case of the unilateral obstacle problem can be found in \cite[Lem.~3.7]{RaulsUlbrich}.

For the bilateral case,
characterizations of the G\^ateaux derivative are summarized 
in the following theorem.

\begin{theorem}
	\label{Theo:GateauxDerivative}
	Assume the obstacles $\psi, \varphi$ satisfy \cref{AssumptionObstacles}.
	Suppose $S_f$ is G\^ateaux differentiable at $u \in U$. Let $h \in U$.
	Then $S_f'(u;h)$ is determined by the solution of the variational equation
	\begin{align}
	\label{VEGateaux}
	\text{Find } \eta \in H_0^1(D):
	\quad
	\langle L\eta-f'(u;h),z\rangle=0
	\quad \forall\, z \in H_0^1(D)
	\end{align}
	and $D$ can be chosen as any quasi-open subset of $\Omega$
	fulfilling
	$I(u) \subseteq D \subseteq \Omega \setminus A_{\mathrm{s}}(u)$.
\end{theorem}

\begin{proof}
	Assume $S_f$ is G\^ateaux differentiable at $u \in U$.
	Then the map $S_f'(u;\cdot)$ is linear
	and the image is a linear subspace of $H_0^1(\Omega)$.
	By the characterization in \cref{DirectionalDerivative},
	the image of $S_f'(u;\cdot)$ lies in a linear subspace of
	the critical cone $(LS_f(u)-f(u))^\perp\cap T_{K_\psi^\varphi}(S_f(u))$.
	The structure of the critical cone established in \cref{Lem:CriticalCone} implies that
	$S_f'(u;h)\in H_0^1(I(f(u)))$,
	since $H_0^1(I(f(u)))$ is the largest linear subset contained in the critical cone.
	Now $S_f'(u;h)$ solves the variational equation \cref{VEGateaux}
	with $D=I(f(u))$, since $H_0^1(I(f(u)))$ is a linear subspace.
	
	Obviously, the image of $S_f'(u;\cdot)$ is also contained
	in the linear hull of the critical cone,
	the set $H_0^1(\Omega \setminus A_{\mathrm{s}}(f(u)))$.
	Assume $z$ is in the critical cone and $h \in U$ is arbitrary.
	Then
	\begin{align*}
	\langle LS_f'(u;-h)-f'(u;-h),z-S_f'(u;-h)\rangle \geq 0,
	\end{align*}
	which implies
	\begin{align*}	
	\langle LS_f'(u;h)-f'(u;h),-z-S_f'(u;h)\rangle \geq 0.
	\end{align*}
	Thus, by linearity arguments
	we obtain that we can also use test functions from the negative critical cone.
	Let $z \in H_0^1(\Omega \setminus A_{\mathrm{s}}(f(u)))$ be arbitrary.
	Since the two sets $\Omega \setminus (A_{\mathrm{s}}(f(u))\cup A_\psi(f(u)))$
	and $\Omega \setminus (A_{\mathrm{s}}(f(u))\cup A^\varphi(f(u)))$
	are a quasi-covering of $\Omega \setminus A_{\mathrm{s}}(f(u))$,
	we can find a sequence $(z_\psi^n+z^\varphi_n)_{n \in \mathbb{N}}$
	converging to $z$ and fulfilling
	$z_\psi^n \in H_0^1(\Omega \setminus (A_{\mathrm{s}}(f(u))\cup A^\varphi(f(u))))$
	and 
	$z_n^\varphi \in H_0^1(\Omega \setminus (A_{\mathrm{s}}(f(u))\cup A_\psi(f(u))))$,
	see \cref{Lem:Capacity}. 
	Considering positive and negative parts we write 
	\begin{align*}
	z_n^\varphi=z_n^{\varphi,+}-z_n^{\varphi,-}\quad
	\text{ and }
	z_\psi^n=z_{\psi,+}^n-z_{\psi,-}^n.
	\end{align*}
	The representation in \cref{StructureCriticalCone} implies that $z_{\psi,+}^n$, $z_{\psi,-}^n$, $-z_n^{\varphi,+}$ and $-z_n^{\varphi,-}$ are elements of the critical cone.
	This shows 
	\begin{align*}
	\langle LS_f'(u;h)-f'(u;h),z_\psi^n+z_n^\varphi-S_f'(u;h)\rangle \geq 0
	\end{align*}
	for all $n \in \mathbb{N}$.
	Taking the limit and observing that $H_0^1(\Omega \setminus A_{\mathrm{s}}(f(u)))$ is a linear subspace we obtain
	\begin{align*}
	\langle LS_f'(u;h)-f'(u;h),z\rangle = 0
	\end{align*}
	for all $z \in H_0^1(\Omega \setminus A_{\mathrm{s}}(f(u)))$.
	
	Thus, since 
	\cref{VEGateaux} is a characterization of the G\^ateaux derivative
	for $D=I(f(u))$ and for $D=\Omega \setminus A_{\mathrm{s}}(f(u))$,
	each set $H_0^1(D)$
	with $I(f(u)) \subseteq D \subseteq \Omega \setminus A_{\mathrm{s}}(f(u))$
	also yields a characterization for the G\^ateaux derivative.
\end{proof}

\section{Monotonicity of the Active and Strictly Active Sets}
\label{Sec:Monotonicity}

In this subsection,
the monotonicity of the active and strictly active sets is studied.
Within this section,
we specify our notation and write
$S_{\psi,f}^\varphi$ instead of $S_f$ for the solution operator of \cref{BOP}.

The monotonicity of the active sets
is a direct consequence of \cref{Lem:Monotonicity}.

\begin{lemma}
	Let $u_1, u_2 \in U$ with $u_1 \geq_U u_2$. Then
	\begin{enumerate}
		\item $A_\psi(f(u_1))\subseteq A_\psi(f(u_2))$,
		\item $A^\varphi(f(u_1)) \supseteq A^\varphi(f(u_2))$.
	\end{enumerate}
\end{lemma}

\begin{proof}
	By \cref{Lem:Monotonicity},
	we have $S_{\psi,f}^\varphi(u_1) \geq S_{\psi,f}^\varphi(u_2)$ a.e.\
	and by \cref{Lem:Capacity} also q.e.\ in $\Omega$.
	This implies the statements.
\end{proof}

\begin{lemma}
	\label{Lem:Invariance1}
	Suppose the conditions of \cref{AssumptionObstacles} are satisfied.
	Let $u \in U$ and let 
	$v \in H_0^1(\Omega)_+$ 
	such that $\{v>0\} \subseteq \Omega \setminus A_\psi^{\mathrm{s}}(f(u))$.
	Then $S_{\psi,f}^\varphi(u)=S_{\psi-v,f}^\varphi(u)$.
\end{lemma}

\begin{proof}
	Obviously, $S_{\psi,f}^\varphi(u) \geq \psi-v$.
	Now, let $z\in K_{\psi-v}^\varphi$ be arbitrary.
	We need to show that
	\begin{align*}
	\langle LS_{\psi,f}^\varphi(u)-f(u),z-S_{\psi,f}^\varphi(u)\rangle \geq 0.
	\end{align*}
	Now, $z=\max(z,\psi)+\min(z-\psi,0)=:z_1+z_2$, where $z_1\in
	K_\psi^\varphi$ and $z_2\in H_0^1(\Omega \setminus A_\psi^{\mathrm{s}}(f(u)))_-$.
	Thus
	\begin{align*}
	\langle LS_{\psi,f}^\varphi(u)-f(u),z_1-S_{\psi,f}^\varphi(u)\rangle \geq 0
	\end{align*}
	and moreover, since $z_2=0$ q.e.\ on $A_\psi^{\mathrm{s}}(f(u))$ 
	and thus $\tilde{\xi}_\psi$-a.e.,
	see \cref{Lem:SignedMeasure},
	we have by \cref{Lem:0OnStrictlyActiveSet}
	\begin{align*}
	\langle LS_{\psi,f}^\varphi(u)-f(u),z_2\rangle=
	\int_{A_\psi} z_2 \,\diff\tilde{\xi}_\psi-\int_{A^\varphi} z_2\,\diff\tilde{\xi}^\varphi
	=-\int_{A^\varphi} z_2\,\diff\tilde{\xi}^\varphi\ge 0.
	\end{align*}
	Here, $\tilde{\xi}_\psi$ and $\tilde{\xi}^\varphi$ are the measures 
	as in \cref{Lem:SignedMeasure}
	that belong to the functional
	$LS_{\psi,f}^\varphi(u)-f(u)=LS_{\psi,\operatorname{id}}^\varphi(f(u))-f(u)$.
\end{proof}

\begin{lemma}
	\label{Lem:ExchangedRole}
	Let $\zeta \in H^{-1}(\Omega)$.
	Then we have
	$-S_{\psi,\operatorname{id}}^\varphi(\zeta)=S_{-\varphi,\operatorname{id}}^{-\psi}(-\zeta)$.
	Moreover, $A_\psi(\zeta)=\tilde{A}^{-\psi}(-\zeta)$ and
	$A^\varphi(\zeta)=\tilde{A}_{-\varphi}(-\zeta)$.
	Here, 
	$\tilde{A}_{-\varphi}(\zeta):=\{\omega \in \Omega \mid S_{-\varphi,\operatorname{id}}^{-\psi}(\zeta)(\omega)=-\varphi(\omega)\}$
	and
	$\tilde{A}^{-\psi}(\zeta):=\{\omega \in \Omega \mid S_{-\varphi,\operatorname{id}}^{-\psi}(\zeta)(\omega)=-\psi(\omega)\}$
	denote the respective active sets for $S_{-\varphi,\operatorname{id}}^{-\psi}(\zeta)$.\\
	Furthermore,
	if the conditions of \cref{AssumptionObstacles} are fulfilled,
	we have
	$A_\psi^{\mathrm{s}}(\zeta)=\tilde{A}_{\mathrm{s}}^{-\psi}(-\zeta)$ and
	$A_{\mathrm{s}}^\varphi(\zeta)=\tilde{A}_{-\varphi}^{\mathrm{s}}(-\zeta)$.
	Here,
	$\tilde{A}_{\mathrm{s}}^{-\psi}(\zeta), \tilde{A}_{-\varphi}^{\mathrm{s}}(\zeta)$ denote the strictly active sets for $S_{-\varphi,\operatorname{id}}^{-\psi}(\zeta)$.
\end{lemma}

\begin{proof}
	First, let us note that for $z \in H_0^1(\Omega)$
	it holds $\psi \leq z \leq \varphi$
	if and only if
	$-\varphi \leq -z \leq -\psi$ and this implies
	\begin{align*}
	-K_\psi^\varphi=K_{-\varphi}^{-\psi}.
	\end{align*}
	Thus, $-S_{\psi,\operatorname{id}}^\varphi(\zeta) \in K_{-\varphi}^{-\psi}$.
	Let now $z \in K_{-\varphi}^{-\psi}$ be arbitrary.
	Then we have
	\begin{align*}
	&\langle L(-S_{\psi,\operatorname{id}}^\varphi(\zeta))+\zeta,z-(-S_{\psi,\operatorname{id}}^\varphi(\zeta))\rangle\\
	&=-\langle LS_{\psi,\operatorname{id}}^\varphi(\zeta)-\zeta,z-(-S_{\psi,\operatorname{id}}^\varphi(\zeta))\rangle\\
	&=\langle LS_{\psi,\operatorname{id}}^\varphi(\zeta)-\zeta,-z-S_{\psi,\operatorname{id}}^\varphi(\zeta)\rangle\\
	&\geq 0.
	\end{align*}
	This yields $-S_{\psi,\operatorname{id}}^\varphi(\zeta)=S_{-\varphi,\operatorname{id}}^{-\psi}(-\zeta)$.
	
	Now, $S_{-\varphi,\operatorname{id}}^{-\psi}(-\zeta)(\omega)=-\psi(\omega)$
	if and only if $S_{\psi,\operatorname{id}}^\varphi(\zeta)(\omega)=\psi(\omega)$
	and we have $S_{-\varphi,\operatorname{id}}^{-\psi}(-\zeta)(\omega)=-\varphi(\omega)$
	if and only if $S_{\psi,\operatorname{id}}^\varphi(\zeta)(\omega)=\varphi(\omega)$,
	thus $A_\psi(\zeta)=\tilde{A}^{-\psi}(-\zeta)$
	and $A^\varphi(\zeta)=\tilde{A}_{-\varphi}(-\zeta)$.
	
	In addition, we have
	\begin{align*}
	LS_{\psi,\operatorname{id}}^\varphi(\zeta)-\zeta=-(LS_{-\varphi,\operatorname{id}}^{-\psi}(-\zeta)-(-\zeta)).
	\end{align*}
	This shows the statements for the strictly active sets.
\end{proof}

\begin{remark}
	If $f \colon U \to H^{-1}(\Omega)$
	satisfies $f(-u)=-f(u)$
	for all $u \in U$,
	then $-S_{\psi,f}^\varphi(u)=S_{-\psi,f}^{-\varphi}(-u)$.
\end{remark}

Now, we check the monotonicity of the strictly active sets.
\begin{lemma}
	\label{Lem:MonotonicityStrictlyActiveSets}
	Let the requirements of \cref{AssumptionObstacles} be satisfied and
	let $u_1 \geq_U u_2$. Then it follows
	\begin{enumerate}
		\item $A_\psi^{\mathrm{s}}(f(u_1))\subseteq A_\psi^{\mathrm{s}}(f(u_2))$,
		\item $A_{\mathrm{s}}^\varphi(f(u_1)) \supseteq A_{\mathrm{s}}^\varphi(f(u_2))$.
	\end{enumerate}
\end{lemma}

\begin{proof}
	%	\begin{enumerate}
	%		\item
	ad 1.:
	Let 
	\begin{align}
	\label{Neighborhood}
	\mathcal{U}=\{S_{\psi,f}^\varphi(u_1)-\psi<(\varphi-\psi)/2\}.
	\end{align}
	Then $\mathcal{U}$
	is quasi-open and 
	$A_\psi^{\mathrm{s}}(f(u_1))\subseteq A_\psi(f(u_1))\subseteq \mathcal{U}\subseteq I^\varphi(f(u_1))$ holds.
	
	Assume $\mathcal{U}
	\setminus A_\psi^{\mathrm{s}}(f(u_2))\neq\emptyset$ (otherwise the assertion follows directly).
	Fix $v \in H_0^1(\mathcal{U})_+$ satisfying $\{v>0\}=\mathcal{U}
	\setminus A_\psi^{\mathrm{s}}(f(u_2))$,
	$v<(\varphi-\psi)/2$, see \cref{Lem:Capacity}.
	
	Let $y_v(t)=S_{\psi-tv, f}^\varphi(u_1)$, $t\in [0,1]$, and denote
	$\bar{y}_v(t):=y_v(t)+tv$.
	Then
	\begin{align*}
	&\langle Ly_v(t)-f(u_1),z-y_v(t)\rangle \geq 0 \quad \forall\,z \in
	K_{\psi-tv}^\varphi\\
	\iff &\langle L\bar{y}_v(t)-f(u_1)-tLv,\bar{z}-\bar{y}_v(t)\rangle \geq 0 \quad
	\forall \,\bar{z}\in K_\psi^{\varphi+tv}\\
	\implies &\langle L\bar{y}_v(t)-f(u_1)-tLv,\bar{z}-\bar{y}_v(t)\rangle \geq 0 \quad
	\forall \,\bar{z}\in K_\psi^{\varphi}.
	\end{align*}
	Now $\psi\le \bar{y}_v(t)
	= y_v(t)+tv\le y_v(0)+tv\leq\varphi$
	by \cref{lem:monobst} and the
	definition of $\mathcal{U}$ and $v$.
	
	Hence, $\bar{y}_v(t)\in K_\psi^{\varphi}$ and
	the last line shows that $y_v(t)=S_{\psi,\operatorname{id}}^\varphi(T(tv))-tv$ with
	$T \colon H_{0}^1(\Omega) \to H^{-1}(\Omega),
	v \mapsto f(u_1)+Lv$.
	Since $S_{\psi,\operatorname{id}}^\varphi$ is directionally differentiable in the Hadamard sense, 
	we can apply the chain rule for the directional derivatives
	and obtain
	\begin{align*}
	y_v'(0;1)&=(S_{\psi,\operatorname{id}}^\varphi)'(T(0);T'(0;v))-v
	=(S_{\psi,\operatorname{id}}^\varphi)'(f(u_1);Lv)-v.
	\end{align*}
	
	Since $(S_{\psi,\operatorname{id}}^\varphi)'(f(u_1);Lv)$ is $0$
	q.e.\ on the strictly active set $A_{\mathrm{s}}(f(u_1))$,
	compare \cref{Subsec:DifferentiabilityProperties} and,
	in particular,
	\cref{Lem:CriticalCone},
	we have $y_v'(0;1)<0$ 
	q.e.\ on $A_{\mathrm{s}}(f(u_1))\cap \{v>0\}$.
	
	Thus, by reducing the lower obstacle on a subset of $A_\psi^{\mathrm{s}}(f(u_1))$
	the solution with respect to the new obstacle will drop on this set.
	
	Now, we show the statement of the lemma by contradiction.
	Therefore, 
	assume the set $W \subseteq \Omega$ is a set of positive capacity
	which is (lower) weakly active for $u_2$ 
	and (lower) strictly active for $u_1$,
	more precisely,
	$W \subseteq A_\psi^{\mathrm{s}}(f(u_1)) \subseteq A_\psi(f(u_1))\subseteq A_\psi(f(u_2))$
	and $W \subseteq \Omega \setminus A_\psi^\mathrm{s}(f(u_2))$
	(up to a set of capacity zero).
	Then $\mathcal{U}$ as in \cref{Neighborhood} is
	a quasi-open neighborhood of $W$
	contained in $I^\varphi(f(u_1))$.
	
	Let as above $v \in H_0^1(\mathcal{U})_+$ satisfy $\{v>0\}=\mathcal{U} \setminus A_\psi^{\mathrm{s}}(f(u_2))$.
	Then, \cref{Lem:Invariance1} yields
	\begin{align} 	
	\label{Vergleich1}
	S_{\psi-v,f}^\varphi(u_2)=S_{\psi,f}^\varphi(u_2)
	\end{align}
	and on $W$ we have
	\begin{align}
	\label{Vergleich2}
	S_{\psi-v,f}^\varphi(u_1)|_{W}<S_{\psi,f}^\varphi(u_1)|_{W}
	=S_{\psi,f}^\varphi(u_2)|_{W}
	\end{align}
	by the structure of the directional derivative
	with respect to the obstacle.
	Putting \cref{Vergleich1,Vergleich2} together, we see that
	\begin{align*}
	S_{\psi-v,f}^\varphi(u_2)>S_{\psi-v,f}^\varphi(u_1)
	\end{align*}
	on $W$.
	On the other hand, $S_{\psi-v,f}^\varphi(u_1)\geq S_{\psi-v,f}^\varphi(u_2)$ since $u_1 \geq_U u_2$.
	Thus,
	such a set $W$ cannot exist and we conclude $A_\psi^{\mathrm{s}}(u_1) \subseteq A_\psi^{\mathrm{s}}(u_2)$.\\
	
	ad 2.:
	By \cref{Lem:ExchangedRole},
	we have $A_{\mathrm{s}}^\varphi(f(u_i))=\tilde{A}_{-\varphi}^{\mathrm{s}}(-f(u_i))$
	for $i=1,2$,
	where we use a similar notation as in \cref{Lem:ExchangedRole}.
	Now, the first part of the lemma implies the statement, since
	\begin{align*}
	A_{\mathrm{s}}^{\varphi}(f(u_1))=
	\tilde{A}_{-\varphi}^{\mathrm{s}}(-f(u_1))\supseteq \tilde{A}_{-\varphi}^{\mathrm{s}}(-(f(u_2))
	=A_{\mathrm{s}}^{\varphi}(f(u_2)).
	\end{align*}
	%	\end{enumerate}
\end{proof}

\section{Mosco Convergence}\label{Sec:Mosco}

For the rest of the paper,
we use again the notation $S_f$ for the solution operator of \cref{BOP}.

The following definition goes back to \cite{Mosco}.
In this form,
the definition can be found, e.g., 
in \cite[Ch.~4:4]{Rodrigues}.

\begin{definition}
	We say that a sequence $(C_n)_{n \in \mathbb{N}}$ 
	of nonempty, closed, convex subsets of a Banach space $X$
	converges to a set $C \subseteq X$ in the sense of Mosco
	if the following two conditions hold.
	\begin{enumerate}
		\item 
		For each $x \in C$ 
		there exists a sequence $(x_n)_{n \in \mathbb{N}}$
		such that $x_n \in C_n$ holds for every $n \in \mathbb{N}$
		and such that $x_n \to x$ in $X$.
		\item
		For each subsequence $(x_{n_k})_{k \in \mathbb{N}}$ of a sequence
		$(x_n)_{n \in \mathbb{N}} \subseteq X$
		fulfilling $x_n \in C_n$ for all $n \in \mathbb{N}$
		such that for some $x \in X$
		we have $x_{n_k} \rightharpoonup x$ in $X$,
		the weak limit $x$ is in $C$.
	\end{enumerate}
\end{definition}

Based on this definition,
the following result on convergence of solutions of
variational inequalities can be established.
It is taken from \cite[Thm.~4.1]{Rodrigues},
see also \cite[Prop.~35.]{Mosco}.
\begin{lemma}
	\label{Lem:RodriguesMosco}
	Assume $L \in \mathcal{L}(H_0^1(\Omega),H^{-1}(\Omega))$ is coercive
	and let $C_n$ and $C$ be nonempty closed, convex subsets of $H_0^1(\Omega)$, $n \in \mathbb{N}$,
	such that $C_n \to C$ in the sense of Mosco.
	Furthermore, let $(h_n)_{n \in \mathbb{N}}, h \subseteq H^{-1}(\Omega)$ with $h_n \to h$.
	Then the solutions of 
	\begin{align*}
	\text{Find } \eta_n \in C_n:\quad
	\langle L\eta_n-h_n,z-\eta_n\rangle \geq 0
	\quad \forall\, z \in C_n
	\end{align*}
	converge strongly in $H_0^1(\Omega)$ to the solution of
	\begin{align*}
	\text{Find } \eta \in C:\quad
	\langle L\eta-h,z-\eta\rangle \geq 0
	\quad \forall\, z \in C.
	\end{align*}
\end{lemma}

Based on this tool,
in order to obtain the convergence of the G\^ateaux derivatives,
which can be characterized as solutions to variational equations,
see \cref{Theo:GateauxDerivative},
and in order to characterize the limit,
we establish the Mosco convergence of the sets $H_0^1(D_n)$.
Depending on either the choice $D_n=I(f(u_n))\cup A_{\mathrm{w}}^\varphi(f(u_n))$
or $D_n=I(f(u_n))\cup A_\psi^{\mathrm{w}}(f(u_n))$,
i.e., 
depending on whether we focus on the inactive set or the complement of the strictly active set
with respect to either the upper or the lower obstacle,
sequences with different monotone behavior have to be considered.

\begin{lemma}
	\label{Lem:MoscoConvergence}
	Suppose \cref{AssumptionObstacles} is satisfied.
	\begin{enumerate}
		\item 
		Let $(u_n)_{n \in \mathbb{N}} \subseteq U$ 
		be an increasing sequence with $u_n \to u$.
		Then $H_0^1(I(f(u_n)) \cup A_{\mathrm{w}}^\varphi(f(u_n))) \to H_0^1(I(f(u))\cup A_{\mathrm{w}}^\varphi(f(u)))$ in the sense of Mosco.
		
		\item
		Let $(u_n)_{n \in \mathbb{N}} \subseteq U$
		be a decreasing sequence with $u_n \to u$.
		Then $H_0^1(I(f(u_n)) \cup A_\psi^{\mathrm{w}}(f(u_n))) \to H_0^1(I(f(u))\cup A_\psi^{\mathrm{w}}(f(u)))$ in the sense of Mosco.
	\end{enumerate}
\end{lemma}

\begin{proof}
	%	\begin{enumerate}
	%		\item 
	ad 1.:
	Assume $v \in H_0^1(I(f(u))\cup A_{\mathrm{w}}^\varphi(f(u)))$ and w.l.o.g.\ $v \geq 0$.
	We can rewrite the function space as
	\begin{align*}
	&H_0^1(I(f(u))\cup A_{\mathrm{w}}^\varphi(f(u)))\\
	&=\{z \in H_0^1(\Omega)\mid z=0 \text{ q.e.\ on } A_\psi(f(u)) \text{ and } z=0 \text{ q.e.\ on } A_{\mathrm{s}}^\varphi(f(u))\}.
	\end{align*}
	Since $A_{\mathrm{s}}^\varphi(f(u_n)) \subseteq A_{\mathrm{s}}^\varphi(f(u))$ for all $n \in \mathbb{N}$,
	see \cref{Lem:MonotonicityStrictlyActiveSets},
	it holds $v=0$ q.e.\ on $A_{\mathrm{s}}^\varphi(f(u_n))$ for all $n \in \mathbb{N}$.
	Since $S_f(u_n) \to S_f(u)$ in $H_0^1(\Omega)$ by continuity of $S_f$,
	we have $S_f(u_n) \to S_f(u)$ for a subsequence pointwise quasi everywhere,
	see \cref{Lem:Capacity}.
	This means
	\begin{align*}
	\operatorname{cap}\left(I_\psi(f(u)) \setminus \bigcup_{k \in \mathbb{N}}I_\psi(f(u_k))\right)=0,
	\end{align*}
	i.e., $(I_\psi(f(u_n)))_{n \in \mathbb{N}}$ is a quasi-covering of $I_\psi(f(u))$, which is increasing in $n$.
	We can therefore find nonnegative $v_n \in H_0^1(\Omega)$ with $v_n \to v$ and $v_n=0$ q.e.\ on $A_\psi(f(u_n))$,
	see \cref{Lem:Capacity}.
	By setting 
	\begin{align*}
	z_n:=\min(v_n,v)
	\end{align*}
	we have $z_n \in H_0^1(I(f(u_n))\cup A_{\mathrm{w}}^\varphi(f(u_n)))$ 
	for all $n \in \mathbb{N}$
	as well as $z_n \to v$.
	
	Let $v_n \in H_0^1(I(f(u_n))\cup A_{\mathrm{w}}^\varphi(f(u_n)))$ for all $n \in \mathbb{N}$.
	%and again, w.l.o.g., $v_n \geq 0$ for all $n \in \mathbb{N}$.
	Assume there is a subsequence $(v_{n_k})_{k \in \mathbb{N}}$
	with $v_{n_k} \rightharpoonup v$ for some $v \in H_0^1(\Omega)$ as $k \to \infty$.
	Since $A_\psi(f(u)) \subseteq A_\psi(f(u_n))$ for all $n \in \mathbb{N}$,
	we conclude $v \in H_0^1(I_\psi(f(u)))$ by Mazur's lemma.
	By \cref{Lem:0OnStrictlyActiveSet} and \cref{Lem:SignedMeasure},
	from $v_n=0$ q.e.\ on $A_{\mathrm{s}}^\varphi(f(u_n))$ and $v_n=0$ q.e.\ on $A_\psi(f(u_n))$ it follows 
	\begin{align*}
	\langle LS_f(u_n)-f(u_n),|v_n|\rangle=\int_{\Omega} |v_n| \,\diff\tilde{\xi}_\psi^n-\int_{\Omega} |v_n| \,\diff\tilde{\xi}_n^\varphi=0
	\end{align*}
	for all $n \in \mathbb{N}$.
	Here, $\tilde{\xi}_\psi^n$ and $\tilde{\xi}_n^\varphi$ are the measures 
	as in \cref{Lem:SignedMeasure}
	that belong to the functional
	$LS_{\psi,f}^\varphi(u_n)-f(u_n)=LS_{\psi,\operatorname{id}}^\varphi(f(u_n))-f(u_n)$.
	From $v_{n_k} \rightharpoonup v$ in $H_0^1(\Omega)$
	we conclude $|v_{n_k}| \rightharpoonup |v|$ in $H_0^1(\Omega)$. 
	(To see this, one can use the compact embedding $H_0^1(\Omega) \hookrightarrow L^2(\Omega)$
	and $\||z|\|_{H_0^1(\Omega)} = \|z\|_{H_0^1(\Omega)}$ for all $z \in H_0^1(\Omega)$, see \cite[Cor.~5.8.1]{AttouchButtazzoMichaille}.)
	Since also $LS_f(u_{n_k})-f(u_{n_k}) \to LS_f(u)-f(u)$ in $H^{-1}(\Omega)$
	we conclude
	\begin{align*}
	0=\langle LS_f(u)-f(u),|v|\rangle=-\int_\Omega |v|\,\diff\tilde{\xi}^\varphi
	\end{align*}
	since $v=0$ q.e.\ on $A_\psi(f(u))$.
	Finally,
	$v=0$ $\tilde{\xi}^\varphi$-a.e.\ on $\Omega$
	and thus
	$v=0$ q.e.\ on $A_{\mathrm{s}}^\varphi(f(u))$, see \cref{Lem:0OnStrictlyActiveSet}.\\
	
	ad 2.:
	Again, this part of the lemma follows from 
	the first part of the lemma combined with
	\cref{Lem:ExchangedRole}.		
	%	\end{enumerate}
\end{proof}	

\section{Generalized Derivatives for the Bilateral Obstacle Problem}\label{Sec:GenDer}

In this section,
we will find a characterization of two generalized derivatives for the solution operator $S_f$
of \cref{BOP}.
To establish this result, we impose the monotonicity assumption \cref{Assumption} on 
$U$ and $f$ stated in the introduction.

As already indicated in the introduction, the assumptions posed on the positive cone in $V$
will ensure that we can construct monotone convergent sequences in $U$ where the G\^ateaux differentiability
of the locally Lipschitz continuous solution operator $S_f$
can be guaranteed.
The tool that is used is the following generalization of Rademacher's theorem
to infinite dimensions,
see e.g.\ \cite[Ch.~II,~Sect.2,~Thm.~1]{Aronszajn}, \cite[Thm.~6.42]{BenyaminiLindenstrauss}.
If the space $X$ in \cref{Thm:Rademacher} is additionally a Hilbert space,
a version can be found in \cite[Thm.~1.2]{Mignot}.

\begin{theorem}
	\label{Thm:Rademacher}
	Assume $T \colon X \to Y$ is locally Lipschitz continuous
	from a separable Banach space $X$
	to a Hilbert space $Y$.
	Then the set $\mathcal{D}_T$ of points at which $T$ is G\^ateaux differentiable
	is a dense subset of $X$.
\end{theorem}

In \cite{Aronszajn},
the map $T$ is Lipschitz continuous and defined on an open subset of $X$.
By considering neighborhoods of points separately, 
the formulation as in \cref{Thm:Rademacher} can be obtained.

Now, we can formulate the main theorem of this paper.

\begin{theorem}
	\label{Theo:MainResult}
	Suppose that \cref{AssumptionObstacles} and \cref{Assumption} are satisfied and 
	let $u \in U$ be arbitrary.
	Then a Bouligand generalized derivative for $S_f$
	in $u$ is given by
	the operator $\Xi(u;\cdot) \in \mathcal{L}(U,H_0^1(\Omega))$,
	where $\Xi(u;h)$ is the unique solution
	of the variational equation
	\begin{align}
	\label{VEGeneralizedDerivative}
	\text{Find } \eta \in H_0^1(D):\quad
	\langle L\eta-f'(u;h),z\rangle = 0
	\quad \forall\, z \in H_0^1(D).
	\end{align}
	Here,
	the sets
	\begin{align*}
	D:=I(f(u))\cup A_{\mathrm{w}}^\varphi(f(u)) \quad \text{ and } \quad D:=I(f(u))\cup A_\psi^{\mathrm{w}}(f(u))
	\end{align*}
	can be chosen and result in generally different generalized derivatives.
\end{theorem}

\begin{proof}
	The proof of \cite[Prop.~5.5]{RaulsUlbrich}
	implies that
	we can find an increasing,
	respectively, decreasing,
	sequence $(u_n)_{n \in \mathbb{N}} \subseteq U$
	such that $S_f$
	is G\^ateaux differentiable at each $u_n$
	and such that $u_n$ converges to $u$.
	Here,
	\cref{Thm:Rademacher} is used.
	
	Let us first assume $(u_n)_{n \in \mathbb{N}}$ is an increasing sequence
	with these properties.
	Let $h \in U$ be arbitrary.
	By \cref{Theo:GateauxDerivative}, 
	for each $n \in \mathbb{N}$,
	$S_f'(u_n;h)$ can be written as the solution
	of the variational equation
	\begin{align}
	\label{VESequenceGateaux}
	\text{Find } \eta_n \in H_0^1(D_n):
	\quad
	\langle L\eta_n-f'(u_n;h),z\rangle=0
	\quad \forall\, z \in H_0^1(D_n)
	\end{align}
	with the choice $D_n:=I(f(u_n))\cup A_{\mathrm{w}}^\varphi(f(u_n))$.
	
	By \cref{Lem:MoscoConvergence},
	we conclude 
	\begin{align*}
	H_0^1(I(f(u_n))\cup A_{\mathrm{w}}^\varphi(f(u_n)))\to H_0^1(I(f(u))\cup A_{\mathrm{w}}^\varphi(f(u)))
	\end{align*}
	in the sense of Mosco. 
	Now,
	\cref{Lem:RodriguesMosco}
	implies that
	$(S_f'(u_n;h))_{n \in \mathbb{N}}$
	converges to the solution of \cref{VEGeneralizedDerivative}
	with $D:=I(f(u))\cup A_{\mathrm{w}}^\varphi(f(u))$.
	By definition,
	the resulting operator is a generalized derivative for $S_f$ in $u$.
	
	When considering a decreasing sequence $(u_n)_{n \in \mathbb{N}}$,
	we use the representation of $S_f'(u_n;h)$ as the solution of the variational equation \cref{VESequenceGateaux}
	with the choice $D_n:=I(f(u_n))\cup A_\psi^{\mathrm{w}}(f(u_n))$
	and obtain the respective Mosco convergence,
	and thus the convergence of $S_f'(u_n)$ to the solution operator of \cref{VEGeneralizedDerivative} with $D:=I(f(u))\cup A_\psi^{\mathrm{w}}(f(u))$ 
	from the second part of \cref{Lem:MoscoConvergence}.
\end{proof}

\subsection{Adjoint Representation of Clarke Subgradients}\label{Sec:Adj}

As in the unilateral case,
see \cite[Thm.~5.7]{RaulsUlbrich},
we can find an adjoint representation 
for the subgradient 
of a reduced objective function.

Assume $J \colon H_0^1(\Omega)\times U\to \mathbb{R}$
is a continuously differentiable objective function.
We consider the optimization problem
\begin{align*}
&\min_{y,u} J(y,u)\\
&\text{subject to } y \in K_\psi^\varphi, 
\quad
\langle Ly-f(u), z-y\rangle\geq 0 \quad \forall\, z \in K_\psi^\varphi. 
\end{align*}
We present a formula for two generalized derivatives
contained in Clarke's generalized differential
that can be obtained
for the reduced objective function
\begin{align*}
\hat{J}(u):=J(S_f(u),u)
\end{align*}
in an arbitrary point $u \in U$.

\begin{corollary}
	\label{Cor:Adjoint}
	Suppose that \cref{AssumptionObstacles} and \cref{Assumption} are satisfied
	and let $u \in U$ be arbitrary.
	Denote by $q$ the unique solution
	of the variational equation
	\begin{align}
	\label{*}
	\text{Find } q \in H_0^1(D),
	\quad
	\langle L^*q,v\rangle=\langle J_y(S_f(u),u),v\rangle
	\quad \forall\, v \in H_0^1(D).
	\end{align}	
	Then the element
	\begin{align*}
	f'(u)^*q+J_u(S_f(u),u)
	\end{align*}
	is in Clarke's generalized differential 
	$\partial_C\hat{J}(u)$.
	In \cref{*},
	the respective sets
	\begin{align*}
	D:=I(f(u))\cup A_{\mathrm{w}}^\varphi(f(u)) \quad \text{ or } \quad D:=I(f(u))\cup A_\psi^{\mathrm{w}}(f(u))
	\end{align*}
	can be chosen and result in a 
	particular generalized derivative.
\end{corollary}

\begin{proof}
	The proof is similar to the proof of \cite[Thm.~5.4]{RaulsUlbrich}.
	
	Since $L^*$ is coercive,
	\cref{*} has a unique solution.
	As stated in \cref{Rem:InclusionGeneralizedDifferentials},
	we have
	\begin{align}
	\label{InclusionSubdifferentials}
	\partial_C \hat{J}(u) \ni \Xi^*J_y(S_f(u),u)+J_u(S_f(u),u)
	\end{align}
	for all $\Xi \in \partial S_f(u)$.
	
	Assume that $q$ solves \cref{*} for $D:=I(f(u))\cup A_{\mathrm{w}}^\varphi(f(u))$, 
	respectively $D:=I(f(u))\cup A_\psi^{\mathrm{w}}(f(u))$.
	For $h \in U$, denote by $\Xi(u;h)$ the solution to \cref{VEGeneralizedDerivative}.
	Now, we have
	\begin{align*}
	\langle f'(u)^*q,w\rangle_{U^*,U}
	&=\langle f'(u;w),q\rangle\\
	&\stackrel{\cref{VEGeneralizedDerivative}}{=}\langle L^*q,\Xi(u;w)\rangle\\
	&\stackrel{\cref{*}}{=}\langle J_y(S_f(u),u),\Xi(u;w)\rangle\\
	&=\langle \Xi(u;\cdot)^*J_y(S_f(u),u),w\rangle_{U^*,U}
	\end{align*}
	for all $w \in U$.
	Since $\Xi(u;\cdot) \in \partial S_f(u)$,
	the statement follows from \cref{InclusionSubdifferentials}.
\end{proof}

\section*{Acknowledgments}
This work was supported by the DFG under grant UL158/10-1 within the Priority Program SPP 1962 
Non-smooth and Complementarity-based Distributed Parameter Systems: Simulation and Hierarchical Optimization.

\bibliography{LiteratureTwosidedVI}
\bibliographystyle{alpha}	
\end{document}